%% file: main.tex
\documentclass[12pt,a4paper,oneside,reqno]{amsproc}
\usepackage{amsmath}
\usepackage{amssymb,amsfonts,mathrsfs}
\usepackage[colorlinks=true,linkcolor=blue,citecolor=blue]{hyperref}
\usepackage[driver=pdftex,lmargin=2.5cm,rmargin=2.5cm,heightrounded=true,centering]{geometry}
\usepackage{math50}
\usepackage{local}
\usepackage{epsfig}
\usepackage{graphicx}
\usepackage{xy}
\xyoption{all}
\def\XXint#1#2#3{{\setbox0=\hbox{$#1{#2#3}{\int}$}
     \vcenter{\hbox{$#2#3$}}\kern-.5\wd0}}

\begin{document}
\title[Peak Solutions for the fractional Nirenberg problem]
{Peak Solutions for the fractional Nirenberg problem}
\author{Yan-Hong Chen}
\address{School of Mathematical Science, Nankai University, Tianjin 300071, P.R. China}
\email{cyh1801@163.com}
\author{Youquan Zheng}
\address{School of Science, Tianjin University, Tianjin 300072, P. R. China.}
\email{zhengyq@tju.edu.cn}
\newcommand{\optional}[1]{\relax}
\setcounter{secnumdepth}{3}
\setcounter{section}{0} \setcounter{equation}{0}
\numberwithin{equation}{section}
\newcommand{\MLversion}{1.1}
\thanks{The second author was partially supported by NSFC of China (11301374) and SRFDP.}
\date{\today}
\begin{abstract}
In this paper, the fractional order curvature equation $(-\Delta)^\gamma u  = (1 + \varepsilon K(x))u^{\frac{N + 2\gamma}{N - 2\gamma}}$ in $\mathbb{R}^N$ is considered. Assuming $K(x)$ has two critical points satisfying certain local conditions, we prove the existence of two-peak solutions.
\end{abstract}
\maketitle
\input{Introduction}

\input{Preliminaries}
\input{Expansion}
\input{Proof}
\begin{appendix}
\input{AppendixA}

\input{AppendixB}
\input{AppendixC}
\input{AppendixD}

\end{appendix}
\bibliography{mrabbrev,mlabbr2003-0,localbib}
\bibliographystyle{plain}
\end{document}

%% file: Introduction.tex
\section{Introduction}
On the standard sphere $(\mathbb{S}^N, g_0)$ of dimension $N\geq 2$, the famous Nirenberg problem asks: whether there exists a function $w$ on $\mathbb{S}^N$ such that the scalar curvature (Gauss curvature for $N = 2$) $R_g$ of the conformal metric $g = e^w g_0$ equals to a prescribed function $K$? Equivalently, one has to solve equations
\begin{equation}\label{e:Yamabe1}
-\Delta_{g_0}w + 1 = K e^w,\quad\text{ on }\mathbb{S}^2
\end{equation}
and
\begin{equation}\label{e:Yamabe}
-\Delta_{g_0}v + \frac{N - 2}{4(N-1)}R_{g_0}v = \frac{N - 2}{4(N - 1)} K v^{\frac{N + 2}{N - 2}},\quad\text{ on }\mathbb{S}^N,\quad N\geq 3,
\end{equation}
where $v = e^{\frac{N-2}{4}w}$.

The linear operators appearing on the left of (\ref{e:Yamabe1}) and (\ref{e:Yamabe}) are the well known conformal Laplacians of $g_0$. Their definition can be extended to a general compact Riemannian manifold $(M, g)$ of dimension $N\geq 2$ which we denote as $P_1^{g}$. In the 1980s, Paneitz discovered another conformally covariant operator
\begin{equation*}\label{e:Paneitz}
P_2^g = (-\Delta_g)^2 - {\rm div}_g (a_N R_gg + b_NRic_g)d + \frac{N -4}{2}Q_N^g,
\end{equation*}
see \cite{Paneitz2008} and \cite{DjadliHebeyLedoux2000}. Here $Q_N^g$ is the $Q$-curvature and $Ric_g$ is the Ricci curvature of $g$, $a_N$, $b_N$ are suitable constants depending on $N$. In \cite{GrahamJenneSparling1992}, a sequence of conformally covariant elliptic operators $\{P_k^g\}$ was constructed for all $k\in \mathbb{N}$ if $N$ is odd but only for $k\in\{1,\cdot\cdot\cdot, N/2\}$ if $N$ is even. In \cite{LiYanyanNonlinearAnalysis2012} and \cite{LiYanyanNonlinearAnalysis2014}, the authors provided a complete characterization for fully nonlinear conformally covariant differential operators of any integer order on $\mathbb{R}^N$. Then the existence of conformally covariant differential operators of non-integer orders is a natural problem. In \cite{Peterson2000}, an intrinsically defined conformally covariant pseudo-differential operator of arbitrary real number order was constructed. In \cite{GramZworski2003}, Graham and Zworski proved that $P_k^g$ can be defined as the residues at $\gamma = k$ of a meromorphic family of
scattering operators and thus a family of elliptic
pseudodifferential operators $P_{\gamma}^g$ for non-integer $\gamma$ was given.

The fractional Laplacian $(-\Delta)^\gamma$ on $\mathbb{R}^{N}$ with $\gamma\in (0, 1)$ is a classical nonlocal operator.
It can be defined as
$$
(-\Delta)^\gamma u = C(N, \gamma)\mbox{P. V.} \int_{\mathbb{R}^N}\frac{u(x) - u(y)}{|x - y|^{N + 2\gamma}}dy = C(N, \gamma)\lim_{\varepsilon\to 0^+}\int_{B_{\varepsilon}^c(x)}\frac{u(x) - u(y)}{|x - y|^{N + 2\gamma}}dy.
$$
Here $C(N, \gamma) = \pi^{-(2\gamma + N/2)}\frac{\Gamma(N/2 + \gamma)}{\Gamma(-\gamma)}$ and P. V. means the principal value.
In the celebrated paper \cite{Caffarelli&Silvestre07} by Caffarelli and Silvestre, the authors proved that a solution $u = u(x)$ of $(-\Delta)^\gamma u = f$ in $\mathbb{R}^{N}$ can be extended to a solution $u = u(x, t)$ of the following problem
\begin{equation*}
\begin{cases} \text{div}(t^{1 -2\gamma}\nabla u) = 0 , &  \text{ in }\mathbb{R}^{N+1}_{+},\\
-d_{\gamma}t^{1 -2\gamma}\partial_{t}u|_{t\to 0} = f, & \text{ on $\partial\mathbb{R}^{N+1}_{+}$},
\end{cases}
\end{equation*}
with $d_{\gamma} = 2^{2\gamma -1}\frac{\Gamma(\gamma)}{\Gamma(1 - \gamma)}$. In \cite{ChangGonzalez2011}, Chang and Gonzalez characterized $P_{\gamma}^g$ as a Dirichlet to Neumann operator on a conformally compact Einstein manifold by the localization method of \cite{Caffarelli&Silvestre07}.

The conformally covariant property for $P_\gamma^g$ with $\gamma\in (0, \frac{N}{2})$ can be expressed as follows:
if $g = v^{\frac{4}{N - 2\gamma}}g_0$, then
\begin{equation}\label{e:conformalltransform}
P_{\gamma}^{g_0}(vf) = v^{\frac{N + 2\gamma}{N - 2\gamma}}P_{\gamma}^g(f)
\end{equation}
holds for any smooth function $f$, see \cite{ChangGonzalez2011}. Similar to the formula for scalar curvature and the Paneitz-Branson $Q$-curvature, the $Q$-curvature for $g$ of order $2\gamma$, can be defined as
\begin{equation*}
Q_{\gamma}^{g} = P_{\gamma}^{g}(1).
\end{equation*}
Thus, on a smooth compact Riemannian manifold $(M, g)$ of dimension $N\geq 2$, one has a question:
is there a representation $g$ of the conformal class $[g]$ such that $Q_{\gamma}^{g} = $ a prescribed function $K$?
By (\ref{e:conformalltransform}), one has to solve the following semi-linear equation,
\begin{equation}\label{e:main111}
P_{\gamma}^{g} v = Kv^{\frac{N + 2\gamma}{N - 2\gamma}}, \,\,\, v > 0 {\rm \,\,\,on\,\,\,} M.
\end{equation}

If $\gamma = 1$, (\ref{e:main111}) is the classical prescribing curvature problem. There is large mount of
research on this equation, we limit ourselves to citing \cite{ChLiMNEE} and \cite{JinLiyanyanXiongNirenbergproblemBlowupanalysis},
referring to their bibliography for more works.
When $\gamma$ is non-integer, we refer the interested readers to \cite{FangGonzalezpreprint}, \cite{GonzalezJGA2012}, \cite{GonzalezAPDE2013}, \cite{QingRaske2006} and the references therein for recent progress. If $(M, g)$ is the standard sphere $(\mathbb{S}^N, g_0)$, (\ref{e:main111}) is the fractional Nirenberg problem
\begin{equation}\label{e:mainllllllllllllll}
\left
\{
\begin{array}{lll}
P_\gamma^{g_0} v = K v^{\frac{N + 2\gamma}{N - 2\gamma}}\text{ on }\mathbb{S}^{N},\\
v > 0,
\end{array}
\right.
\end{equation}
which was studied in
\cite{AbdelhediChtiouiJFA2013}, \cite{JinLiyanyanXiongNirenbergproblemBlowupanalysis}, \cite{JinLiyanyanXiongNirenbergproblemBlowupanalysisII} and \cite{JinLiyanyanXiongNirenbergproblemUnifiedapproach}. In \cite{AbdelhediChtiouiJFA2013}, existence of positive solutions has been proved under the Euler-Hopf-type criterion for $K(x)$ via the method of critical points at infinity (see \cite{BahriACriticalPIVC1988} and \cite{BahriCoron1988}) when $N = 2$. In \cite{JinLiyanyanXiongNirenbergproblemBlowupanalysis}, blow up and compactness results of solutions were obtained. In \cite{JinLiyanyanXiongNirenbergproblemBlowupanalysisII}, existence of solutions was proved under conditions of Bahri-Coron type and a fractional Aubin inequality was also proved. In \cite{JinLiyanyanXiongNirenbergproblemUnifiedapproach}, a unified approach was developed for blowing up profiles, compactness and existence of solutions for (\ref{e:mainllllllllllllll}). We also note that critical problems on domains of $\mathbb{R}^N$ involving the fractional Laplacian were studied in \cite{PalatucciNA2015} and \cite{PalatucciCVPDE2014}.

Although the operator $P_\gamma^g$ on a general manifold depends on the extension problem, the operator $P_\gamma^{g_0}$ is the $2\gamma$ order conformal Laplacian on $\mathbb{S}^{N}$ and it can be uniquely expressed as
$$
P_\gamma^{g_0} = \frac{\Gamma(B + \frac{1}{2} +\gamma)}{\Gamma(B + \frac{1}{2} -\gamma)},\quad B = \sqrt{-\Delta_{g_0} + \left(\frac{N - 1}{2}\right)^2}.
$$
Let $\mathcal{N}$ be the north pole of $\mathbb{S}^N$ and
\begin{equation*}
F:\mathbb{R}^N\to \mathbb{S}^N\setminus\{\mathcal{N}\},\quad x\mapsto \left(\frac{2x}{1 + |x|^2}, \frac{|x|^2 - 1}{|x|^2 + 1}\right)
\end{equation*}
be the inverse of stereographic projection operator from $\mathbb{S}^N\setminus\{\mathcal{N}\}$ to $\mathbb{R}^N$. Then one has the following relation
\begin{equation*}\label{e:Qcurvature}
P_{\gamma}^{g_0}(\phi)\circ F = |J_F|^{-\frac{N + 2\gamma}{2N}}(-\Delta)^{\gamma}(|J_F|^{\frac{N - 2\gamma}{2N}}(\phi\circ F)){\rm\,\,\, for\,\, all\,\,\,}\phi\in C^{\infty}(\mathbb{S}^N),
\end{equation*}
where $|J_F| = \left(\frac{2}{1 + |x|^2}\right)^N$. Hence, for a solution $v$ of (\ref{e:mainllllllllllllll}),
$u(x) = |J_F|^{\frac{n - 2\gamma}{2n}}v(F(x))$ satisfies
\begin{equation}\label{e:mainllllllllll}
(-\Delta)^\gamma u = (Q\circ F) u^{\frac{N + 2\gamma}{N - 2\gamma}},\,\,\, u > 0 {\rm\,\,\,on\,\,\,}\mathbb{R}^N.
\end{equation}

In this paper, let $Q\circ F$ in (\ref{e:mainllllllllll}) be of form $Q\circ F = 1 + \varepsilon K(x)$
and we are interested in peak solutions for this equation. That is to say, we consider
\begin{equation}\label{e:main}
\begin{cases} (-\Delta)^\gamma u  = (1 + \varepsilon K(x))u^{\frac{N + 2\gamma}{N - 2\gamma}}\quad x\in\mathbb{R}^N,\\
u > 0\quad\text{in $\mathbb{R}^N$},
\end{cases}
\end{equation}
where $\varepsilon > 0$, $\gamma \in (0, 1)$, $2\gamma < N$ and $K\in C^1(\mathbb{R}^N)\cap L^\infty(\mathbb{R}^N)$.

Let $S_{\gamma}$ be the best constant of Sobolev embeddings
\begin{equation*}
\dot{H}^\gamma(\mathbb{R}^N)\hookrightarrow L^{2^*_\gamma}(\mathbb{R}^N),
\end{equation*}
where $2^*_\gamma = 2N/(N - 2\gamma)$. So the following inequality holds
\begin{equation}\label{e:embedding}
\displaystyle S_{\gamma}\left[\int_{\mathbb{R}^N}|u|^{2^*_\gamma}dx\right]^{\frac{2}{2^*_\gamma}}\leq \int_{\mathbb{R}^N}|(-\Delta)^{\frac{\gamma}{2}}u|^2dx\quad\text{for all } u\in\dot{H}^\gamma(\mathbb{R}^N).
\end{equation}
From a celebrated result of Lieb in \cite{LiebAnnals1983}, one has that
\begin{equation*}
S_{\gamma} = \displaystyle\left(2^{-2\gamma}\pi^{-\gamma}\frac{\Gamma(\frac{N - 2\gamma}{2})}{\Gamma(\frac{N + 2\gamma}{2})}\left[\frac{\Gamma(N)}{\Gamma(N/2)}\right]^{2\gamma/N}\right)^{-\frac{2^*_\gamma}{2}}.
\end{equation*}
Moreover, the extremals are of form
\begin{equation}\label{e:extremals}
U_{x_0, \lambda}(x) = C_0\left(\frac{\lambda}{1 + \lambda^2|x - x_0|^2}\right)^{\frac{N - 2\gamma}{2}}\quad \forall x\in \mathbb{R}^N,
\end{equation}
where $C_0\in \mathbb{R}\setminus\{0\}$, $\lambda > 0$ are constants and $x_0\in \mathbb{R}^N$ is a fixed point.
In \cite{ChenLiCPAM2006} and \cite{LiYanyanJEMS2005}, the authors proved that all positive solutions of the Euler-Lagrange equation of the best constant problem for (\ref{e:embedding}) are of form (\ref{e:extremals}). More precisely, they proved that every positive regular solution for the semilinear partial differential equation
\begin{equation*}\label{e:unperturbed}
(-\Delta)^\gamma u = u^{\frac{N + 2\gamma}{N - 2\gamma}}
\end{equation*}
assumes the form
\begin{equation*}
U_{x_0, \lambda}(x) = C_0\left(\frac{\lambda}{1 + \lambda^2|x - x_0|^2}\right)^{\frac{N - 2\gamma}{2}},\quad \forall x\in \mathbb{R}^N
\end{equation*}
for some constant $C_0 = C_0(N, \gamma)$.

Let $\Sigma$ be the set of all critical points of $K$ (we denote such critical points as $z$), satisfying (after a suitable change of the coordinate system depending on $z$),
$$
K(x) = K(z) + \displaystyle\sum_{i = 1}^N a_i|x_i - z_i|^\beta + O(|x - z|^{\beta + \sigma})
$$
for $x$ close to $z$, where $a_i$, $\beta$ and $\sigma$ are some constants depending on $z$, $a_i\neq 0$ for $i = 1,\cdot\cdot\cdot, N$, $\sum_{i = 1}^N a_i < 0$, $\beta\in (1, N - 2\gamma)$ and $\sigma\in (0, 1)$. Then we have the main result of our paper.
\begin{theorem}\label{t:main}
Suppose that $\gamma \in (0, 1)$, $2\gamma < N$, $K\in C^1(\mathbb{R}^N)\cap L^\infty(\mathbb{R}^N)$ and $\Sigma$ contains at least two points. Then for each pair $z^1$, $z^2\in \Sigma$, $z^1\neq z^2$, there exits an $\varepsilon_0 > 0$ such that when $\varepsilon\in (0, \varepsilon_0)$, (\ref{e:main}) has at least a solution of form
\begin{equation*}\label{e:solution}
u_\varepsilon = \sum_{j = 1}^2\alpha_{j, \varepsilon} U_{y^j_\varepsilon, \lambda_{j,\varepsilon}} + v_\varepsilon, \text{ with } \alpha_{j, \varepsilon}\to 1, y^j_\varepsilon \to z^j, \lambda_{j, \varepsilon}\to +\infty
\end{equation*}
and $\|v_\varepsilon\|_{\dot{H}^\gamma(\mathbb{R}^N)}\to 0$ as $\varepsilon\to 0$.
\end{theorem}

When $\gamma = 1$, Theorem \ref{t:main} was obtained in \cite{CaoNoussairYan2002}. In \cite{ChenZheng2014}, the author studied problem (\ref{e:main}) and extended a result of \cite{AmbrosettiGarciaPeralJFA1999}. Both these two papers imposed a condition relating all the critical points in $\Sigma$ which implies that $\Sigma$ contains only finite number of points. The condition used in Theorem \ref{t:main} is a local one and we only assume that $\Sigma$ has at least two points. If $\Sigma$ contains more than two points, our result also gives a multiplicity result. Also note that, the solutions obtained in \cite{ChenZheng2014} have energy close to $\frac{1}{N}S^{\frac{N}{2\gamma}}_\gamma$ for small $\varepsilon$ and those solutions need not concentrate at some point. But the corresponding energy of solutions in Theorem \ref{t:main} is close to $\frac{2}{N}S^{\frac{N}{2\gamma}}_\gamma$ when $\varepsilon$ is small and these solutions must concentrate at $z^1$ and $z^2$. To prove the main theorem, we will use arguments similar to \cite{CaoNoussairYan2002}, see also \cite{CaoPeng2013}, \cite{CaoPengYanWSCPCR2013}, \cite{PengshuangjieZhouJing2010} and \cite{Yanshusen2000}.

The rest of this paper is organized as follows. In Section 2, we introduce some preliminary knowledge and the variational structure. In Section 3, we expand the functional and its gradient. The basic estimates needed in the proof are given in the Appendices. In Section 4, we give the proof of the Theorem \ref{t:main}.

%% file: Preliminaries.tex
\section{Preliminaries and variational structure}
Fractional Sobolev spaces have played an important role in harmonic analysis, functional analysis and partial differential equations. We refer the interested readers to \cite{DiNezzaPalatucciValdinoci2012} for an elementary introduction to this topic and a wide list of related references.  Corresponding to the fractional Sobolev spaces, fractional powers of the Laplacian operator and related nonlocal equations arise in numerous of problems both from mathematics and its concrete applications, see, for example, \cite{AbdelouhabBonaFellandSaut1989}, \cite{ChangGonzalez2011}, \cite{ElgartSchleinCPAM2007}, \cite{FrohlichLenzmanCPAM2007}, \cite{KenigMartelRobiano2011}, \cite{LiebYau1987}, \cite{MajdaMclaughlinTabak1997},  \cite{SavinValdini2012} and \cite{Weistein1987} and the references therein.

For $\gamma\in (0, 1)$ and $N\geq 2$, the fractional Sobolev space $H^\gamma(\mathbb{R}^N)$ is defined by
\begin{equation*}
H^\gamma(\mathbb{R}^N) = \displaystyle\left\{u\in L^2(\mathbb{R}^N): \frac{|u(x) - u(y)|}{|x -y|^{N/2 + \gamma}}\in L^2(\mathbb{R}^N\times\mathbb{R}^N)\right\}
\end{equation*}
endowed with the natural norm
\begin{equation*}
\|u\|_{H^\gamma(\mathbb{R}^N)} := \displaystyle\left(\int_{\mathbb{R}^N}|u|^2dx + \int_{\mathbb{R}^N}\int_{\mathbb{R}^N}\frac{|u(x) - u(y)|^2}{|x -y|^{N + 2\gamma}}dxdy\right)^{\frac{1}{2}}.
\end{equation*}
The so-called Gagliardo semi-norm of $u$ is defined by
\begin{equation*}
[u]_{H^\gamma(\mathbb{R}^N)}  := \displaystyle\left(\int_{\mathbb{R}^N}\int_{\mathbb{R}^N}\frac{|u(x) - u(y)|^2}{|x -y|^{N + 2\gamma}}dxdy\right)^{\frac{1}{2}}.
\end{equation*}
We also define $\dot{H}^\gamma(\mathbb{R}^N)$ be the homogeneous version of $H^\gamma(\mathbb{R}^N)$ as the completion of $C^{\infty}_0(\mathbb{R}^N)$ under the Gagliardo semi-norm. Let $H^\gamma_0(\mathbb{R}^N)$ denote the closure of $C^{\infty}_0(\mathbb{R}^N)$ under the the norm $\|\cdot\|_{H^\gamma(\mathbb{R}^N)}$. Note that, in view of Theorem 2.4 in \cite{DiNezzaPalatucciValdinoci2012}, $H^\gamma_0(\mathbb{R}^N)=H^\gamma(\mathbb{R}^N)$.

Let $\mathcal{S}$ be the Schwartz space of rapidly decaying smooth functions on $\mathbb{R}^{N}$. The topology of $\mathcal{S}$ is generated by the seminorms
$$
p_{m}(\varphi) = \displaystyle\sup_{x\in\mathbb{R}^{N}}(1 + |x|)^{m}\sum_{|\gamma|\leq m}|D^{\gamma}\varphi(x)|,\quad m = 0, 1, 2,\cdot\cdot\cdot,
$$
where $\varphi\in \mathcal{S}$. Denote the topological dual of $\mathcal{S}$ by $\mathcal{S}'$. For any $\varphi\in \mathcal{S}$,
$$
\mathcal{F}\varphi(\xi) = \frac{1}{(2\pi)^{N/2}}\int_{\mathbb{R}^{N}}e^{-i\xi\cdot x}\varphi(x)dx
$$
is the usual Fourier transformation of $\varphi$ and one can extend $\mathcal{F}$ from $\mathcal{S}$ to $\mathcal{S}'$. The relation between the fractional operator $(-\Delta)^\gamma$ and the fractional Sobolev space $H^{\gamma}(\mathbb{R}^{N})$ is yield in the following identity,
$$
[u]_\gamma = C\left(\int_{\mathbb{R}^N}|\xi|^{2\gamma}|\mathcal{F}u(\xi)|^2d\xi\right)^{\frac{1}{2}} = C\|(-\Delta)^{\frac{\gamma}{2}}u\|_{L^2(\mathbb{R}^N)}
$$
for a suitable positive constant $C=C(N, \gamma)$.

For $y = (y^1, y^2)\in\mathbb{R}^N\times\mathbb{R}^N$, $\lambda = (\lambda_1, \lambda_2)\in \mathbb{R}^2$, we define the space
\begin{eqnarray*}
E_{y, \lambda}^2 = \{v\in \dot{H}^\gamma(\mathbb{R}^N)|\displaystyle\langle U_{y^j, \lambda_j}, v\rangle &=& \langle\frac{\partial U_{y^j, \lambda_j}}{\partial \lambda_j}, v \rangle = \langle\frac{\partial U_{y^j, \lambda_j}}{\partial y^j_i}, v\rangle = 0,\\ \text{ for } j &=& 1,\cdot\cdot\cdot,2,\quad i = 1, \cdot\cdot\cdot, N\}.
\end{eqnarray*}
Our aim is to find two-peak solutions $u_\varepsilon$ for (\ref{e:main}) of form
\begin{equation*}
u_\varepsilon = \alpha_{1, \varepsilon}U_{y^1_\varepsilon, \lambda_{1, \varepsilon}} + \alpha_{2, \varepsilon}U_{y^2_\varepsilon, \lambda_{2, \varepsilon}} + v_\varepsilon,
\end{equation*}
with $\alpha_\varepsilon = (\alpha_{1, \varepsilon}, \alpha_{2, \varepsilon})\to (1, 1)$, $y_\varepsilon = (y^1_\varepsilon, y^2_\varepsilon)\to (z^1, z^2)$, $\lambda_{j, \varepsilon}\to +\infty$ for $j = 1, 2$ and $v_\varepsilon \in E_{y_\varepsilon, \lambda_\varepsilon}^2$ satisfying $\|v_\varepsilon\|_{\dot{H}^\gamma(\mathbb{R}^N)}\to 0$ as $\varepsilon\to 0$.

For each $z^1$, $z^2\in \Sigma$, $z^1\neq z^2$, $\mu > 0$, let us define
\begin{eqnarray*}
D_\mu = \{(y, \lambda)\,\,|\,\, y = (y^1, y^2)\in \overline{B_\mu(z^1)}\times \overline{B_\mu(z^2)}, \lambda = (\lambda_1, \lambda_2)\in (\frac{1}{\mu}, +\infty)\times (\frac{1}{\mu}, +\infty)\}
\end{eqnarray*}
and
\begin{eqnarray*}
M_\mu &=& \{(\alpha, y, \lambda, v)|\alpha = (\alpha_1, \alpha_2)\in\mathbb{R}^+\times\mathbb{R}^+, (y, \lambda)\in D_\mu, v\in E_{y, \lambda}^2,\\ &&|\alpha_1 - 1|\leq \mu, |\alpha_2 - 1|\leq \mu, \|v\|\leq \mu\}.
\end{eqnarray*}
Furthermore, define
\begin{equation}\label{e:functional}
I_\varepsilon (u) = \frac{1}{2}\displaystyle\int_{\mathbb{R}^N}|(-\Delta)^{\frac{\gamma}{2}} u(x)|^2dx - \frac{N - 2\gamma}{2N}\displaystyle\int_{\mathbb{R}^N}(1 + \varepsilon K)|u|^{\frac{2N}{N - 2\gamma}}dx
\end{equation}
and
\begin{equation*}
J_\varepsilon (\alpha, y, \lambda, v) = I_\varepsilon (\sum_{j = 1}^2\alpha_j U_{y^j, \lambda_j} + v).
\end{equation*}

It is well known that, when $\mu > 0$ is sufficiently small, $(\alpha, y, \lambda, v)\in M_\mu$ is a critical point of $J_\varepsilon$ in $M_\varepsilon$ if and only if $u = \sum_{j = 1}^2\alpha_j U_{y^j, \lambda_j} + v$ is a critical point of $I_\varepsilon$ in $\dot{H}^\gamma(\mathbb{R}^N)$. Moreover, solutions of form $u = \sum_{j = 1}^2\alpha_j U_{y^j, \lambda_j} + v$ for (\ref{e:main}) are positive. For completeness, we will give the proof of these two facts in Lemma \ref{e:criticalpoint} and Lemma \ref{e:positive}. In order to prove the main theorem, we only need to find a critical point of $J_\varepsilon$ in $M_\varepsilon$.

From the Lagrange multiplier theorem, $(\alpha, y, \lambda, v)\in M_\mu$ is a critical point of $J_\varepsilon$ in $M_\varepsilon$ if and only if the following equations are satisfied,
\begin{equation}\label{e:lagrangian1}
\frac{\partial J_\varepsilon}{\partial \alpha_j} = 0,\quad j = 1, 2,
\end{equation}
\begin{equation}\label{e:lagrangian2}
\frac{\partial J_\varepsilon}{\partial v} = \sum_{j = 1}^2A_j U_{y^j, \lambda_j} + \sum_{j = 1}^2B_j \frac{\partial U_{y^j, \lambda_j}}{\partial \lambda_j} + \sum_{j = 1}^2\sum_{i = 1}^NC_{ji}\frac{\partial U_{y^j, \lambda_j}}{\partial y^j_i},
\end{equation}
\begin{equation}\label{e:lagrangian3}
\displaystyle\frac{\partial J_\varepsilon}{\partial y^j_i} = B_j \langle\frac{\partial ^2U_{y^j, \lambda_j}}{\partial \lambda_j\partial y^j_i}, v\rangle + \sum_{l = 1}^NC_{jl}\langle\frac{\partial^2 U_{y^j, \lambda_j}}{\partial y^j_l\partial y^j_i}, v\rangle,\quad i = 1,\cdot\cdot\cdot,N, \quad j = 1,2,
\end{equation}
\begin{equation}\label{e:lagrangian4}
\displaystyle\frac{\partial J_\varepsilon}{\partial \lambda_j} = B_j \langle\frac{\partial ^2U_{y^j, \lambda_j}}{\partial \lambda_j^2}, v\rangle + \sum_{l = 1}^NC_{jl}\langle\frac{\partial^2 U_{y^j, \lambda_j}}{\partial y^j_l\partial \lambda_j}, v\rangle,\quad j = 1,2,
\end{equation}
for some $A_j$, $B_j$, $C_{ji}\in \mathbb{R}$, $j = 1,2$, $i = 1,\cdot\cdot\cdot, N$.

%% file: Expansion.tex
\section{Expansion of $J_\varepsilon$ and its gradient}
Based on the estimates in Appendices B and C, we will give the expansion of $J_\varepsilon$ and its gradient.
Denote $\varepsilon_{12} = \frac{1}{(\lambda_1\lambda_2)^{\frac{N - 2\gamma}{2}}}$.

For $z^1$, $z^2\in \Sigma$ and $(y, \lambda)\in D_\mu$. Let $\hat{\alpha} = ((1 + \varepsilon K(z^1))^{-\frac{N - 2\gamma}{4}}, (1 + \varepsilon K(z^2))^{-\frac{N - 2\gamma}{4}})$, $\bar{\alpha} = \alpha - \hat{\alpha}$, $\omega = (\bar{\alpha}, v)\in \mathbb{R}^2\times E^2_{y, \lambda}$. Then $\hat{J}_{\varepsilon}(y, \lambda, \omega): = J_\varepsilon(\alpha, y, \lambda, v)$ can be expanded by
\begin{equation*}
\hat{J}_{\varepsilon}(y, \lambda, \omega) = \hat{J}_{\varepsilon}(y, \lambda, 0) + \langle f_\varepsilon, \omega\rangle + \frac{1}{2}\langle Q_\varepsilon \omega, \omega\rangle + R_\varepsilon(\omega).
\end{equation*}
Here $f_\varepsilon \in \mathbb{R}^2\times E^2_{y, \lambda}$ is a functional given by
\begin{eqnarray*}
&&\langle f_\varepsilon, \omega\rangle = -\displaystyle\int_{\mathbb{R}^n}(1 + \varepsilon K)\left(\sum_{j = 1}^2\hat{\alpha}_{j} U_{y^j, \lambda_{j}}\right)^{2^*_\gamma - 1}vdx + \displaystyle\sum_{k = 1}^2\bar{\alpha}_k\\ &&\left(\left\langle \sum_{j = 1}^2\hat{\alpha}_{j} U_{y^j, \lambda_{j}}, U_{y^k, \lambda_k}\right\rangle-\int_{\mathbb{R}^n}(1 + \varepsilon K)\left(\sum_{j = 1}^2\hat{\alpha}_{j} U_{y^j, \lambda_{j}}\right)^{2^*_\gamma - 1}U_{y^k, \lambda_k}dx\right),
\end{eqnarray*}
$Q_\varepsilon$ is a quadratic form on $\mathbb{R}^2\times E^2_{y, \lambda}$ defined as
\begin{eqnarray*}
&&\langle Q_\varepsilon \omega, \omega\rangle\\ &&= \displaystyle\sum_{k = 1}^2\bar{\alpha}_k\bar{\alpha}_l\left(\left\langle U_{y^l, \lambda_l}, U_{y^k, \lambda_k}\right\rangle
-(2^*_\gamma - 1)\int_{\mathbb{R}^n}(1 + \varepsilon K)\left(\sum_{j = 1}^2\hat{\alpha}_{j} U_{y^j, \lambda_{j}}\right)^{2^*_\gamma - 2}U_{y^l, \lambda_l}U_{y^k, \lambda_k}dx\right)
\end{eqnarray*}
\begin{eqnarray*}
&&+ \|v\|^2_{\dot{H}^\gamma(\mathbb{R}^N)} -(2^*_\gamma - 1)\displaystyle\int_{\mathbb{R}^n}(1 + \varepsilon K)\left(\sum_{j = 1}^2\hat{\alpha}_{j} U_{y^j, \lambda_{j}}\right)^{2^*_\gamma - 2}v^2dx\\
&&- 2(2^*_\gamma - 1)\displaystyle\sum_{k = 1}^{2}\bar{\alpha}_k\int_{\mathbb{R}^n}(1 + \varepsilon K)\left(\sum_{j = 1}^2\hat{\alpha}_{j} U_{y^j, \lambda_{j}}\right)^{2^*_\gamma - 2}U_{y^k, \lambda_k}vdx
\end{eqnarray*}
and $R_\varepsilon$ is the remainder term satisfying
\begin{equation*}
D^{(i)}R_\varepsilon(\omega) = O(\|\omega\|^{2 + \theta - i}_{\dot{H}^\gamma(\mathbb{R}^N)})\quad i = 0, 1, 2
\end{equation*}
for some constant $\theta > 0$.

From Lemmas \ref{l:lemmaC1} and \ref{l:lemmaC2} in Appendix B, if $\mu > 0$, $\varepsilon > 0$ are small enough, then for each $(y, \lambda)\in D_\mu$, $Q_\varepsilon$ is invertible. Moreover, there exists $C > 0$, which is independent of $(y, \lambda)$, such that $\|Q_\varepsilon^{-1}\|_{\dot{H}^\gamma(\mathbb{R}^N)}\leq C$. From Lemmas \ref{l:lemmaB2}, \ref{l:lemmaB3} and \ref{l:lemmaB4} in Appendix C, we have
$$
|\langle f_\varepsilon, \omega\rangle| = O\left(\varepsilon \sum_{j = 1}^2\left(|y^j - z^j|^{\beta_j} + \frac{1}{\lambda_j^{\beta_j}}\right) + \varepsilon_{12}^{\frac{1}{2} + \tau}\right)\|\omega\|_{\dot{H}^\gamma(\mathbb{R}^N)},
$$
where $\tau > 0$ is some constant.

Consider the equation $D\hat{J}_{\varepsilon}(y, \lambda, \omega) = 0$, where $D$ is the derivative with respect to the variable $\omega$. It is equivalent to
\begin{equation}\label{e:firstderivative}
f_\varepsilon + Q_\varepsilon \omega + DR(\omega) = 0.
\end{equation}
By the implicit function theorem, there exists a unique $C^1$ map $w(y, \lambda)$ satisfying (\ref{e:firstderivative}) and
$$
\|\omega\|_{\dot{H}^\gamma(\mathbb{R}^N)}\leq C\|f_\varepsilon\|_{\dot{H}^\gamma(\mathbb{R}^N)}.
$$
Thus we obtain $\omega(y, \lambda)\in \mathbb{R}^2\times E^2_{y, \lambda}$ such that $D\hat{J}_{\varepsilon}(y, \lambda, \omega) = 0$ for each fixed $(y, \lambda)\in D_\mu$. So we have the following
\begin{lemma}\label{l:lemma2.1}
Suppose $z^1$, $z^2\in \Sigma$ and $(y, \lambda)\in D_\mu$. Then there exist $\varepsilon_0 > 0$ and $\mu_0 > 0$ such that for $\varepsilon\in (0, \varepsilon_0)$ and $\mu\in (0, \mu_0)$, there is a unique $C^1$ map $(y, \lambda)\in D_\mu\to (\alpha_\varepsilon(y, \lambda), v_\varepsilon(y, \lambda))\in \mathbb{R}^2\times \dot{H}^\gamma(\mathbb{R}^N)$ such that $v_\varepsilon\in E^2_{y, \lambda}$, $(\alpha_\varepsilon(y, \lambda), y, \lambda,  v_\varepsilon(y, \lambda))$ satisfies (\ref{e:lagrangian1}) and (\ref{e:lagrangian2}), $\alpha_\varepsilon = (\alpha_{\varepsilon, 1}, \alpha_{\varepsilon, 2})$ and $v_\varepsilon$ satisfy the following estimate,
\begin{equation*}\label{e:estimate1}
\displaystyle\sum_{j = 1}^2\left|\alpha_{j, \varepsilon} - (1 + \varepsilon K(z^j))^{-\frac{N - 2\gamma}{4}}\right| + \|v_\varepsilon\|_{\dot{H}^\gamma(\mathbb{R}^N)} = O\left(\varepsilon \sum_{j = 1}^2\left(|y^j - z^j|^{\beta_j} + \frac{1}{\lambda_j^{\beta_j}}\right) + \varepsilon_{12}^{\frac{1}{2} + \tau}\right)
\end{equation*}
as $\varepsilon\to 0$, where $\tau > 0$ is some constant.
\end{lemma}
\begin{lemma}\label{l:lemma2.2}
Let $(y, \lambda)\in D_\mu$, $(\alpha, v)$ be obtained as in Lemma \ref{l:lemma2.1}. Then for $\mu > 0$ and $\varepsilon > 0$ small enough, it holds that
\begin{eqnarray*}
\frac{\partial J_\varepsilon (\alpha, y, \lambda, v)}{\partial \lambda_k}& = & C_{N, \beta_k} \frac{\varepsilon}{\lambda_k^{\beta_k + 1}}\displaystyle\sum_{i = 1}^Na_i^k + \frac{C_0\varepsilon_{12}}{\lambda_k|z^1 - z^2|^{N - 2\gamma}} + O\left(\frac{\varepsilon\varepsilon_{12}}{\lambda_k}\right) +
O\left(\frac{\varepsilon_{12}^{1 + \tau_1}}{\lambda_k}\right)\\ && +\,\, O\left(\frac{\varepsilon}{\lambda_k^{\beta_k}}|y^k - z^k|\right)
+ O\left(\frac{\varepsilon}{\lambda_k}\sum_{j = 1}^2\left(\frac{1}{\lambda_j^{\beta_j +\sigma}} + |y^j - z^j|^{\beta_j + \sigma}\right)\right),
\end{eqnarray*}
where $C_{N, \beta_k} > 0$ and $C_0 > 0$ are constants and $\tau_1 = \min\{\tau, \frac{1}{N}\}$.
\end{lemma}
\begin{proof}
Without loss of generality, we assume $k = 1$. By Lemma \ref{l:lemmaB7} and Lemma \ref{l:lemma2.1}, we have
\begin{eqnarray*}
&&\frac{\partial J_\varepsilon (\alpha, y, \lambda, v)}{\partial \lambda_1}=  \alpha_1 \left\langle\sum_{j = 1}^2\alpha_j U_{y^j, \lambda_j} + v, \frac{\partial U_{y^j, \lambda_j}}{\partial \lambda_1}\right\rangle\\ &&\,\,\,\,\,\, - \alpha_1\displaystyle\int_{\mathbb{R}^n}(1 + \varepsilon K)\left|\sum_{j = 1}^2\alpha_{j} U_{y^j, \lambda_{j}} + v\right|^{2^*_\gamma - 2}\left(\sum_{j = 1}^2\alpha_{j} U_{y^j, \lambda_{j}} + v\right)\frac{\partial U_{y^j, \lambda_j}}{\partial \lambda_1}dx\\
&& = \alpha_1\alpha_{2}\displaystyle\int_{\mathbb{R}^n}U_{y^2, \lambda_{2}}^{2^*_\gamma - 1}\frac{\partial U_{y^1, \lambda_1}}{\partial \lambda_1}dx - \alpha_1\displaystyle\int_{\mathbb{R}^n}(1 + \varepsilon K)\left(\sum_{j = 1}^2\alpha_{j} U_{y^j, \lambda_{j}}\right)^{2^*_\gamma - 1}\frac{\partial U_{y^1, \lambda_1}}{\partial \lambda_1}dx\\
&&\,\,\,\,\,\, -\alpha_1(2^*_\gamma - 1)\displaystyle\int_{\mathbb{R}^n}(1 + \varepsilon K)\left(\sum_{j = 1}^2\alpha_{j} U_{y^j, \lambda_{j}}\right)^{2^*_\gamma - 2}\frac{\partial U_{y^1, \lambda_1}}{\partial \lambda_1}vdx + O(\frac{1}{\lambda_1})\|v\|^2_{\dot{H}^\gamma(\mathbb{R}^N)}\\
&& = \alpha_1(\alpha_2 - \alpha_2^{2^*_\gamma - 1})\displaystyle\int_{\mathbb{R}^n}U_{y^2, \lambda_{2}}^{2^*_\gamma - 1}\frac{\partial U_{y^1, \lambda_1}}{\partial \lambda_1}dx - \alpha_1^{2^*_\gamma}\varepsilon\displaystyle\int_{\mathbb{R}^n}K\alpha_{j} U_{y^1, \lambda_{1}}^{2^*_\gamma - 1}\frac{\partial U_{y^1, \lambda_1}}{\partial \lambda_1}dx\\
&&\,\,\,\,\,\, -(2^*_\gamma - 1)\alpha_1^{2^*_\gamma - 1}\alpha_2\displaystyle\int_{\mathbb{R}^n}\alpha_{j} U_{y^1, \lambda_{1}}^{2^*_\gamma - 2}\frac{\partial U_{y^1, \lambda_1}}{\partial \lambda_1}U_{y^2, \lambda_2}dx + O\left(\frac{\varepsilon\varepsilon_{12}}{\lambda_1}\right)\\
&&\,\,\,\,\,\, +\,\,
O\left(\frac{\varepsilon_{12}^{1 + \tau_1}}{\lambda_1}\right)+ O\left(\frac{\varepsilon}{\lambda_1}\sum_{j = 1}^2\left(\frac{1}{\lambda_j^{\beta_j + \sigma}} +  |y^j - z^j|^{\beta_j + \sigma}\right)\right)\\
&& =  -\varepsilon\displaystyle\int_{\mathbb{R}^n}K U_{y^1, \lambda_{1}}^{2^*_\gamma - 1}\frac{\partial U_{y^1, \lambda_1}}{\partial \lambda_1}dx -(2^*_\gamma - 1)\displaystyle\int_{\mathbb{R}^n}U_{y^1, \lambda_{1}}^{2^*_\gamma - 2}\frac{\partial U_{y^1, \lambda_1}}{\partial \lambda_1}U_{y^2, \lambda_2}dx\\ && \,\,\,\,\,\,+\,\, O\left(\frac{\varepsilon\varepsilon_{12}}{\lambda_1}\right) + O\left(\frac{\varepsilon_{12}^{1 + \tau_1}}{\lambda_1}\right)+ O\left(\frac{\varepsilon}{\lambda_1}\sum_{j = 1}^2\left(\frac{1}{\lambda_j^{\beta_j + \sigma}} + |y^j - z^j|^{\beta_j + \sigma}\right)\right).
\end{eqnarray*}
From Lemmas \ref{l:lemmaB8} and \ref{l:lemmaB10}, we have the assertion of this lemma.
\end{proof}
\begin{lemma}\label{l:lemma2.3}
Under the same assumptions of Lemma \ref{l:lemma2.2}, it holds that
\begin{eqnarray*}
\frac{\partial J_\varepsilon (\alpha, y, \lambda, v)}{\partial y^k_i}& = & - D_{N, \beta_k} a_i^k\frac{\varepsilon}{\lambda_k^{\beta_k - 2}}(y^k_i - z^k_i) + O\left(\lambda_k\varepsilon_{12}^{1 + \tau_1}\right) - \frac{C_2(y^k_i - y^l_i)}{(\lambda_1\lambda_2)^{\frac{N - 2\gamma}{2}}}\\ && +\,\,
O\left(\frac{\varepsilon}{\lambda_k^{\beta_k - 1}}\lambda_k^2|y^k - z^k|^2\right) + O\left(\varepsilon\lambda_k\varepsilon_{12}\right) +
O\left(\lambda_k\varepsilon_{12}^{1 + \tau_1}\right)\\ &&
+\,\, O\left(\lambda_k\varepsilon\sum_{j = 1}^2\left(\frac{1}{\lambda_j^{\beta_j +\hat{\sigma}}} + |y^j - z^j|^{\beta_j + \hat{\sigma}}\right)\right) + O\left(\varepsilon_{12}^{\frac{N - \gamma}{N - 2\gamma}}\right),
\end{eqnarray*}
where $D_{N, \beta_k} > 0$ and $C_2 > 0$ are constants, $\tau_1$ and $\hat{\sigma}$ are the same as in Lemma \ref{l:lemma2.2}.
\end{lemma}
\begin{proof}
The proof is similar to Lemma \ref{l:lemma2.2}. Indeed,
\begin{eqnarray*}
&&\frac{\partial J_\varepsilon (\alpha, y, \lambda, v)}{\partial y^1_i}=  \alpha_1 \left\langle\sum_{j = 1}^2\alpha_j U_{y^j, \lambda_j} + v, \frac{\partial U_{y^j, \lambda_j}}{\partial y^1_i}\right\rangle\\ && \,\,\,\,\,\,- \alpha_1\displaystyle\int_{\mathbb{R}^n}(1 + \varepsilon K)\left|\sum_{j = 1}^2\alpha_{j} U_{y^j, \lambda_{j}} + v\right|^{2^*_\gamma - 2}\left(\sum_{j = 1}^2\alpha_{j} U_{y^j, \lambda_{j}} + v\right)\frac{\partial U_{y^j, \lambda_j}}{\partial y^1_i}dx\\
&& = \alpha_1\alpha_{2}\displaystyle\int_{\mathbb{R}^n}U_{y^2, \lambda_{2}}^{2^*_\gamma - 1}\frac{\partial U_{y^1, \lambda_1}}{\partial y^1_i}dx - \alpha_1\displaystyle\int_{\mathbb{R}^n}(1 + \varepsilon K)\left(\sum_{j = 1}^2\alpha_{j} U_{y^j, \lambda_{j}}\right)^{2^*_\gamma - 1}\frac{\partial U_{y^1, \lambda_1}}{\partial y^1_i}dx\\
&& \,\,\,\,\,\,-\alpha_1(2^*_\gamma - 1)\displaystyle\int_{\mathbb{R}^n}(1 + \varepsilon K)\left(\sum_{j = 1}^2\alpha_{j} U_{y^j, \lambda_{j}}\right)^{2^*_\gamma - 2}\frac{\partial U_{y^1, \lambda_1}}{\partial y^1_i}vdx + O(\lambda_1)\|v\|^2_{\dot{H}^\gamma(\mathbb{R}^N)}\\
&& = \alpha_1(\alpha_2 - \alpha_2^{2^*_\gamma - 1})\displaystyle\int_{\mathbb{R}^n}U_{y^2, \lambda_{2}}^{2^*_\gamma - 1}\frac{\partial U_{y^1, \lambda_1}}{\partial y^1_i}dx - \alpha_1^{2^*_\gamma}\varepsilon\displaystyle\int_{\mathbb{R}^n}K\alpha_{j} U_{y^1, \lambda_{1}}^{2^*_\gamma - 1}\frac{\partial U_{y^1, \lambda_1}}{\partial y^1_i}dx\\
&&\,\,\,\,\,\, -(2^*_\gamma - 1)\alpha_1^{2^*_\gamma - 1}\alpha_2\displaystyle\int_{\mathbb{R}^n}\alpha_{j} U_{y^1, \lambda_{1}}^{2^*_\gamma - 2}\frac{\partial U_{y^1, \lambda_1}}{\partial y^1_i}U_{y^2, \lambda_2}dx + O\left(\varepsilon\lambda_1\varepsilon_{12}\right)\\ &&\,\,\,\,\,\,+\,\,
O\left(\lambda_1\varepsilon_{12}^{1 + \tau_1}\right)+ O\left(\lambda_1\varepsilon\sum_{j = 1}^2\left(\frac{1}{\lambda_j^{\beta_j + \sigma}} +  |y^j - z^j|^{\beta_j + \sigma}\right)\right)\\
&& =  -\varepsilon\displaystyle\int_{\mathbb{R}^n}K U_{y^1, \lambda_{1}}^{2^*_\gamma - 1}\frac{\partial U_{y^1, \lambda_1}}{\partial y^1_i}dx -(2^*_\gamma - 1)\displaystyle\int_{\mathbb{R}^n} U_{y^1, \lambda_{1}}^{2^*_\gamma - 2}\frac{\partial U_{y^1, \lambda_1}}{\partial y^1_i}U_{y^2, \lambda_2}dx\\ && \,\,\,\,\,\,+\,\, O\left(\varepsilon\lambda_1\varepsilon_{12}\right)+
O\left(\lambda_1\varepsilon_{12}^{1 + \tau_1}\right)+ O\left(\lambda_1\varepsilon\sum_{j = 1}^2\left(\frac{1}{\lambda_j^{\beta_j + \sigma}} + |y^j - z^j|^{\beta_j + \sigma}\right)\right).
\end{eqnarray*}
By Lemmas \ref{l:lemmaB9} and \ref{l:lemmaB11}, we have the conclusion. This completes the proof.
\end{proof}
\begin{lemma}\label{l:lemma2.4}
For $(y, \lambda)\in D_\mu$, let $(\alpha, v)\in\mathbb{R}^2\times E^2_{y, \lambda}$ be obtained in Lemma \ref{l:lemma2.1}. Then the following estimates hold
\begin{equation*}
B_k = O\left(\lambda_k\varepsilon_{12}\right) + O\left(\lambda_k^2\sum_{j = 1}^2\left(\frac{\varepsilon}{\lambda_j^{\beta_j +1}} + \varepsilon |y^j - z^j|^{\beta_j + 1}\right)\right),
\end{equation*}
\begin{equation*}
C_{ki}= O\left(\frac{1}{\lambda_k^2}\left(\frac{\varepsilon}{\lambda_k^{\beta_k - 1}}\lambda_k|y^k - z^k| + \varepsilon_{12}\right)\right)
+ O\left(\frac{1}{\lambda_k}\sum_{j = 1}^2\left(\frac{\varepsilon}{\lambda_j^{\beta_j - 1 +\sigma}} + \varepsilon\lambda_j|y^j - z^j|^{\beta_j + \sigma}\right)\right),
\end{equation*}
here, $\sigma > 0$ is some constant.
\end{lemma}
\begin{proof}
For $\varphi\in \dot{H}^\gamma(\mathbb{R}^N)$, we have
\begin{equation*}
\langle\frac{\partial J_\varepsilon}{\partial v}, \varphi\rangle = \sum_{j = 1}^2A_j \langle U_{y^j, \lambda_j}, \varphi\rangle + \sum_{j = 1}^2B_j \langle\frac{\partial U_{y^j, \lambda_j}}{\partial \lambda_j}, \varphi\rangle + \sum_{j = 1}^2\sum_{i = 1}^NC_{ij}\langle\frac{\partial U_{y^j, \lambda_j}}{\partial y^j_i}, \varphi\rangle.
\end{equation*}
Choosing $\varphi = U_{y^k, \lambda_k}$, $\frac{\partial U_{y^k, \lambda_k}}{\partial \lambda_k}$, $\frac{\partial U_{y^k, \lambda_k}}{\partial y^k_h}$, $k = 1, 2$, $h = 1,\cdot\cdot\cdot, N$ respectively and using the fact that
$$
\langle\frac{\partial J_\varepsilon}{\partial v}, \frac{\partial U_{y^k, \lambda_k}}{\partial \lambda_k}\rangle =   \frac{\partial J_\varepsilon}{\partial \lambda_k}, \quad\langle\frac{\partial J_\varepsilon}{\partial v}, \frac{\partial U_{y^k, \lambda_k}}{\partial y^k_h}\rangle =   \frac{\partial J_\varepsilon}{\partial y^k_h},
$$
we obtain a system of equations of $A_j$, $B_j$, $C_{ji}$. Moreover, the coefficient matrix is quasi diagonal. Using the estimate in Lemmas \ref{l:lemma2.2}, \ref{l:lemma2.3} and in the estimates in appendix C, we obtain the desired estimates.
\end{proof}

%% file: Proof.tex
\section{Proof of the main result.}
Inspired by \cite{CaoNoussairYan2002}, we set $L_\varepsilon = \varepsilon^{-\frac{\beta_1\beta_2}{\frac{N - 2\gamma}{2}(\beta_1 + \beta_2) - \beta_1\beta_2}}$. By Lemma \ref{l:lemma2.1}, we only need to show that equations (\ref{e:lagrangian3}) and (\ref{e:lagrangian4}) are satisfied by some $(y, \lambda)\in D_\mu$.

By Lemmas \ref{l:lemma2.2}-\ref{l:lemma2.4}, (\ref{e:lagrangian3}) and (\ref{e:lagrangian4}) are equivalent to the following system,
\begin{equation}\label{e:lagrangian5}
\frac{\varepsilon}{\lambda_k^{\beta_k}}\lambda_k(y^k_i - z^k_i) = O\left(\varepsilon\displaystyle\sum_{j = 1}^2\left(\frac{1}{\lambda_j^{\beta_j + \sigma}} + |y^j - z^j|^{\beta_j + \sigma}\right)\right) + O(\frac{\varepsilon_{12}}{\lambda_k}), \quad k = 1,2,\quad i = 1,\cdot\cdot\cdot, N,
\end{equation}
\begin{equation}\label{e:lagrangian6}
\displaystyle\sum_{i = 1}^N a_i^k \frac{\varepsilon}{\lambda_k^{\beta_k}} + \frac{d_k}{(\lambda_1\lambda_2)^{\frac{N - 2\gamma}{2}}} = O(\varepsilon\varepsilon_{12}) + O(\varepsilon_{12}^{1 + \tau_1}) + O\left(\varepsilon\displaystyle\sum_{j = 1}^2\left(\frac{1}{\lambda_j^{\beta_j + \sigma}} + |y^j - z^j|^{\beta_j + \sigma}\right)\right),
\end{equation}
where $d_k > 0$, $\tau_1 > 0$ are constants.

Let
$$
\lambda_1 = t_1 L_\varepsilon^{\beta_1^{-1}}, \quad \lambda_2 = t_2 L_\varepsilon^{\beta_2^{-1}}, \quad t_1, t_2\in [\gamma_1, \gamma_2],
$$
$$
y^1 - z^1 = \lambda_1^{-1}x^1, \quad y^2 - z^2 = \lambda_2^{-1} x^2, \quad x^1, x^2\in B_{\delta}(0).
$$
Then (\ref{e:lagrangian5}) and (\ref{e:lagrangian6}) are equivalent to
\begin{equation}\label{e:lagrangian7}
x^k = o_{\varepsilon}(1)\quad  k = 1, 2,
\end{equation}
\begin{equation}\label{e:lagrangian8}
t_k^{-\beta_k} + \left(\displaystyle\sum_{i = 1}^N a_i^k\right)^{-1}d_k(t_1t_2)^{-\frac{N - 2\gamma}{2}} = o_{\varepsilon}(1), \quad k = 1, 2.
\end{equation}

Let
$$
f(x^1, x^2) = (x^1, x^2), \quad (x^1, x^2)\in \Omega_1:= B_{\delta}(0)\times B_{\delta}(0),
$$
$$
g(t_1, t_2) = (g_1(t_1, t_2), g_2(t_1, t_2)),\quad (t_1, t_2)\in \Omega_2:= [\gamma_1, \gamma_2]\times [\gamma_1, \gamma_2],
$$
$$
g_k(t_1, t_2) = \frac{1}{t_k^{\beta_k}} - \frac{m_k}{(t_1t_2)^{\frac{N - 2\gamma}{2}}}, \quad\text{ where } m_k = - \frac{d_k}{\sum_{i = 1}^N a_i^k} > 0.
$$
Then
\begin{equation}\label{e:proof1}
deg(f, \Omega_1, 0) = 1.
\end{equation}

It is easy to see that $g = 0$ has a unique solution $(t_1^*, t_2^*)$ in $[\gamma_1, \gamma_2]\times [\gamma_1, \gamma_2]$ when $\gamma_1$ is small and $\gamma_2 > 0$ is large enough. Moreover, it holds that,
$$
\frac{\partial g_1}{\partial t_1}|_{(t_1^*, t_2^*)} = \frac{1}{t_1^*}\left(-\beta_1 + \frac{N - 2\gamma}{2}\right)\frac{m_1}{(t_1^*t_2^*)^{\frac{N - 2\gamma}{2}}},
$$
$$
\frac{\partial g_1}{\partial t_2}|_{(t_1^*, t_2^*)} = \frac{1}{t_2^*}\frac{(N - 2\gamma)m_1}{2(t_1^*t_2^*)^{\frac{N - 2\gamma}{2}}},
$$
$$
\frac{\partial g_2}{\partial t_1}|_{(t_1^*, t_2^*)} = \frac{1}{t_1^*}\frac{(N - 2\gamma)m_2}{2(t_1^*t_2^*)^{\frac{N - 2\gamma}{2}}},
$$
$$
\frac{\partial g_1}{\partial t_2}|_{(t_1^*, t_2^*)} = \frac{1}{t_2^*} \left(-\beta_2 + \frac{N - 2\gamma}{2}\right)\frac{m_2}{(t_1^*t_2^*)^{\frac{N - 2\gamma}{2}}}.
$$
Thus the Jacobian determinant of $g$ at $(t_1^*, t_2^*)$ satisfies
$$
J_g|_{(t_1^*, t_2^*)} = (\beta_1\beta_2 - (\beta_ 1 + \beta_ 2)\frac{N - 2\gamma}{2})\frac{m_1m_2}{(t_1^*t_2^*)^{N - \gamma}} < 0.
$$
So
\begin{equation}\label{e:proof2}
deg (g, \Omega_2, 0) = -1.
\end{equation}

From (\ref{e:proof1}) and (\ref{e:proof2}), we have
$$
deg ((f\times g), \Omega_1\times \Omega_2, 0) = deg (f, \Omega_1, 0)\times deg (g, \Omega_2, 0) = -1.
$$
Hence there exists a solution for (\ref{e:lagrangian7}) and (\ref{e:lagrangian8}).
This completes the proof.

%% file: AppendixA.tex
\section{}
In this appendix, we prove that, for $\mu > 0$ sufficiently small, $(\alpha, y, \lambda, v)\in M_\mu$ is a critical point of $J_\varepsilon$ in $M_\varepsilon$ if and only if $u = \sum_{j = 1}^2\alpha_j U_{y^j, \lambda_j} + v$ is a critical point of $I_\varepsilon$ in $\dot{H}^\gamma(\mathbb{R}^N)$ and the fact that  solutions of form $u = \sum_{j = 1}^2\alpha_j U_{y^j, \lambda_j} + v$ for (\ref{e:main}) are positive.

For $x\in \mathbb{R}^N$ and $\lambda > 0$, define
\begin{equation*}
\delta(y, \lambda)(x) = U_{y, \lambda}(x) = C_0\left(\frac{\lambda}{1 + \lambda^2|x - y|^2}\right)^{\frac{N - 2\gamma}{2}}.
\end{equation*}
Following the idea of \cite{BahriCoron1988}, first, we prove
\begin{lemma}\label{l:A1}
Let $(\mu_k)$ be a sequence of positive numbers, $\lim_{k\to +\infty}\mu_k = 0$ and let $(\alpha^k, y_k, \lambda^k)$, $(\tilde{\alpha}^k, \tilde{y}_k, \tilde{\lambda}^k)\in B_{\mu_k} :=\{(\alpha, y, \lambda)|\alpha = (\alpha_1, \alpha_2)\in\mathbb{R}^+\times\mathbb{R}^+, (y, \lambda)\in D_{\mu_k}, |\alpha_1 - 1|\leq \mu_k, |\alpha_2 - 1|\leq \mu_k\}$ be two sequences such that
\begin{equation}\label{e:A1}
\lim_{k\to +\infty}\|\sum_{j = 1}^2\alpha^k_j U_{y_k^j, \lambda^k_j} - \sum_{j = 1}^2\tilde{\alpha}^k_j U_{\tilde{y}_k^j, \tilde{\lambda}^k_j}\|_{\dot{H}^\gamma(\mathbb{R}^N)} = 0.
\end{equation}
Then it holds that
\begin{equation}\label{e:A5}
\lim_{k\to +\infty}(\lambda^k_j/ \tilde{\lambda}^k_j) = 1,\quad \text{for all } j = 1, 2,
\end{equation}
\begin{equation}\label{e:A6}
\lim_{k\to +\infty}\lambda^k_j\tilde{\lambda}^k_j|y_k^j - \tilde{y}_k^j|^2 = 0,\quad \text{for all } j = 1, 2,
\end{equation}
\begin{equation}\label{e:A4}
\lim_{k\to +\infty}|\alpha^k_j - \tilde{\alpha}^k_j| = 0,\quad \text{for all } j = 1, 2.
\end{equation}
\end{lemma}
\begin{proof}
First, it holds that
\begin{equation*}
\|\delta(y, \lambda)\|_{\dot{H}^\gamma(\mathbb{R}^N)} = 1
\end{equation*}
and
\begin{equation}\label{e:A2}
\displaystyle\lim_{\lambda/\lambda' + \lambda'/\lambda + \lambda\lambda'|y - y'|^2 \to +\infty}\int_{\mathbb{R}^N}(-\Delta)^{\frac{\gamma}{2}}\delta(y, \lambda)(-\Delta)^{\frac{\gamma}{2}}\delta(y', \lambda')dx = 0.
\end{equation}
Then by (\ref{e:A1}) and (\ref{e:A2}), we have
\begin{equation}\label{e:A3}
\displaystyle\int_{\mathbb{R}^N}|\alpha^k_i(-\Delta)^{\frac{\gamma}{2}}\delta(y^i_k, \lambda^k_i) - \tilde{\alpha}^k_i(-\Delta)^{\frac{\gamma}{2}}\delta(\tilde{y}^i_k, \tilde{\lambda}^k_i)|^2dx = o(1)\quad\text{ for } i = 1, 2.
\end{equation}
Hence
\begin{equation*}
\lambda^k_i/\tilde{\lambda}^k_i + \tilde{\lambda}^k_i/\lambda^k_i + \lambda^k_i\tilde{\lambda}^k_i|y^k_i - \tilde{y}^k_i|^2\leq c.
\end{equation*}
Set
\begin{equation*}
w = C_0\left(\frac{1}{1 + |x|^2}\right)^{\frac{N - 2\gamma}{2}},
\end{equation*}
then from (\ref{e:A3}), (\ref{e:A4}) holds. Furthermore,
\begin{eqnarray*}
&&\displaystyle\int_{\mathbb{R}^N}|(-\Delta)^{\frac{\gamma}{2}}\omega - (-\Delta)^{\frac{\gamma}{2}}\delta(\lambda_i(y^i_k - \tilde{y}^i_k), \tilde{\lambda}^k_i/\lambda^k_i)|^2dx\\ && =  \int_{\mathbb{R}^N}|(-\Delta)^{\frac{\gamma}{2}}\delta(y^i_k, \lambda^k_i) - (-\Delta)^{\frac{\gamma}{2}}\delta(\tilde{y}^i_k, \tilde{\lambda}^k_i)|^2dx\\ &&= o(1).
\end{eqnarray*}
Thus we have (\ref{e:A5}) and (\ref{e:A6}).
\end{proof}
\begin{lemma}\label{l:existence}
There exists a constant $\mu_0 > 0$ such that when $\mu\in (0, \mu_0]$ and for each $u\in \dot{H}^\gamma(\mathbb{R}^N)$ satisfying
\begin{equation*}
\|u - \sum_{j = 1}^2 U_{\bar{y}^j, \bar{\lambda}_j}\|_{\dot{H}^\gamma(\mathbb{R}^N)}\leq \mu, \quad\text{for some } (\bar{y}, \bar{\lambda})\in D_\mu,
\end{equation*}
the infimum problem
$$
\displaystyle\inf_{(\alpha, y, \lambda)\in B_{4\mu}}\|u - \sum_{j = 1}^2\alpha_j U_{y^j, \lambda_j}\|_{\dot{H}^\gamma(\mathbb{R}^N)}
$$
is achieved in $B_{2\mu}$ and is not achieved in $B_{4\mu}\setminus B_{2\mu}$.
\end{lemma}
\begin{proof}
Let us prove that the infimum cannot be achieved in $B_{4\mu_0}\setminus B_{2\mu_0}$ if $\mu_0$ is small enough. If this is not true, then there exist $\{\mu_k\}$ with $\mu_k > 0$ and $\mu_k = o(1)$, a sequence $(\tilde{\alpha}^k, \tilde{y}_k, \tilde{\lambda}^k)\in B_{4\mu_k}\setminus B_{2\mu_k}$ and $(y_k, \lambda^k)\in D_{\mu_k}$ such that
\begin{equation*}
\|u - \sum_{j = 1}^2 U_{y^j_k, \lambda_j^k}\|_{\dot{H}^\gamma(\mathbb{R}^N)}\leq \mu_k.
\end{equation*}
Then we have
\begin{equation*}
\|\sum_{j = 1}^2\tilde{\alpha}^k U_{\tilde{y}_k^j, \tilde{\lambda}^k_j} - \sum_{j = 1}^2U_{y^j_k, \lambda^k_j}\|_{\dot{H}^\gamma(\mathbb{R}^N)} = o(1).
\end{equation*}
By Lemma \ref{l:A1}, it holds that
\begin{equation}\label{e:A7}
\lambda^k_i/ \tilde{\lambda}^k_i = 1 + o(1),\quad \text{for all } i = 1, 2,
\end{equation}
\begin{equation}\label{e:A8}
\lambda^k_i\tilde{\lambda}^k_i|y_k^i - \tilde{y}_k^i|^2 = o(1),\quad \text{for all } i = 1, 2.
\end{equation}
(\ref{e:A7}), (\ref{e:A8}) and the fact that $(\tilde{\alpha}^k, \tilde{y}_k, \tilde{\lambda}^k)\in B_{4\mu_k}\setminus B_{2\mu_k}$ and $(y_k, \lambda^k)\in D_{\mu_k}$ is a contradiction. The other assertion can be proved similarly.
\end{proof}
\begin{lemma}\label{l:uniqueness}
There is a constant $\mu_0 > 0$ such that when $\mu\in (0, \mu_0]$ and for each $u\in \dot{H}^\gamma(\mathbb{R}^N)$ satisfying
\begin{equation*}
\|u - \sum_{j = 1}^2 U_{\bar{y}^j, \bar{\lambda}_j}\|_{\dot{H}^\gamma(\mathbb{R}^N)}\leq \mu, \quad\text{for some } (\bar{y}, \bar{\lambda})\in D_\mu,
\end{equation*}
the infimum problem
$$
\displaystyle\inf_{(\alpha, y, \lambda)\in B_{4\mu}}\|u - \sum_{j = 1}^2\alpha_j U_{y^j, \lambda_j}\|_{\dot{H}^\gamma(\mathbb{R}^N)}
$$
is uniquely achieved.
\end{lemma}
\begin{proof}
By Lemma \ref{l:existence}, we only need to prove the uniqueness part of the assertion. We argue by contradiction. If the statement is false, then there exist $\{\mu_k\}$ with $\mu_k = o(1)$, $\{u_k\}$ such that
\begin{equation*}
\|u_k - \sum_{j = 1}^2 U_{\bar{y}_k^j, \bar{\lambda}^k_j}\|_{\dot{H}^\gamma(\mathbb{R}^N)}\leq \mu_k, \quad\text{for some } (\bar{y}_k, \bar{\lambda}^k)\in D_{\mu_k},
\end{equation*}
$(\alpha^k, y_k, \lambda^k)$, $(\tilde{\alpha}^k, \tilde{y}_k, \tilde{\lambda}^k)$ in $B_{2\mu_k}$ for which the following properties hold:\\
first,
$$
(\alpha^k, y_k, \lambda^k) \neq (\tilde{\alpha}^k, \tilde{y}_k, \tilde{\lambda}^k);
$$
second, if $v^k = u^k - \sum_{j = 1}^2\alpha^k_j U_{y_k^j, \lambda^k_j}$ and $\tilde{v}^k = u^k - \sum_{j = 1}^2\tilde{\alpha}^k_j U_{{\tilde{y}}_k^j, \tilde{\lambda}^k_j}$, then it holds that
\begin{equation}\label{e:Ua1}
0 = \displaystyle\int_{\mathbb{R}^N}(- \Delta)^{\frac{\gamma}{2}}v^k(- \Delta)^{\frac{\gamma}{2}} \delta_{i}^kdx =  \displaystyle\int_{\mathbb{R}^N}(- \Delta)^{\frac{\gamma}{2}}v^k(- \Delta)^{\frac{\gamma}{2}} \frac{\partial\delta_{i}^k}{\partial \lambda_{i}^k}dx\quad \text{ for } i = 1, 2 \text{ and all } k,
\end{equation}
\begin{equation}\label{e:Ua2}
0 = \displaystyle\int_{\mathbb{R}^N}(- \Delta)^{\frac{\gamma}{2}}v^k(- \Delta)^{\frac{\gamma}{2}} \frac{\partial\delta_{i}^k}{\partial y_{k}^i}dx\quad \text{ for} i = 1, 2 \text{ and all } k,
\end{equation}
\begin{equation}\label{e:Ua3}
0 = \displaystyle\int_{\mathbb{R}^N}(- \Delta)^{\frac{\gamma}{2}}\tilde{v}^k(- \Delta)^{\frac{\gamma}{2}} \tilde{\delta}_{i}^kdx =  \displaystyle\int_{\mathbb{R}^N}(- \Delta)^{\frac{\gamma}{2}}\tilde{v}^k(- \Delta)^{\frac{\gamma}{2}} \frac{\partial\tilde{\delta}_{i}^k}{\partial \lambda_{i}^k}dx\quad \text{ for } i = 1, 2 \text{ and all } k,
\end{equation}
\begin{equation}\label{e:Ua4}
0 = \displaystyle\int_{\mathbb{R}^N}(- \Delta)^{\frac{\gamma}{2}}\tilde{v}^k(- \Delta)^{\frac{\gamma}{2}} \frac{\partial\tilde{\delta}_{i}^k}{\partial \tilde{y}_{k}^i}dx\quad \text{ for all } i \text{ and all } k,
\end{equation}
here
\begin{equation*}
\delta_{i}^k = \delta(y_k^i, \lambda_i^k),\quad \tilde{\delta}_{i}^k = \delta(\tilde{y}_k^i, \tilde{\lambda}_i^k).
\end{equation*}
In the following, we shall omit the index $k$. Then, by Lemma \ref{l:A1}, we have
\begin{equation*}
\lambda_i/\tilde{\lambda}_i = 1 + o(1),
\end{equation*}
\begin{equation*}
\lambda_i\tilde{\lambda}_i|x_i - \tilde{x}_i|^2 = o(1),
\end{equation*}
\begin{equation*}
|\alpha_i - \tilde{\alpha}_i| = o(1).
\end{equation*}
Also by (\ref{e:Ua1}) and (\ref{e:Ua3}), we get
\begin{equation*}
\displaystyle\sum_{j}\int_{\mathbb{R}^N}\left(\alpha_j (- \Delta)^{\frac{\gamma}{2}}\delta_j - \tilde{\alpha}_j (- \Delta)^{\frac{\gamma}{2}}\tilde{\delta}_j\right)(- \Delta)^{\frac{\gamma}{2}}\delta_idx = \int_{\mathbb{R}^N}(- \Delta)^{\frac{\gamma}{2}}\tilde{v}((- \Delta)^{\frac{\gamma}{2}} \delta_i - (- \Delta)^{\frac{\gamma}{2}} \tilde{\delta}_i)dx.
\end{equation*}
Let $\eta_i = \tilde{\lambda}_i/\lambda_i - 1$, $a_i = \tilde{\lambda}_i(x_i - \tilde{x}_i)$, $\mu_i = \alpha_i - \tilde{\alpha}_i$. Note that $|a_i| = o(1)$, $\eta_i = o(1)$, $\mu_i = o(1)$, it is easy to see that
\begin{equation*}
|\tilde{\delta}_j(y) - \delta_j(y)|\leq c(|\eta_j| + |a_j|)\delta_j(y),
\end{equation*}
\begin{eqnarray*}
\int_{\mathbb{R}^N}\left(\alpha_j (- \Delta)^{\frac{\gamma}{2}}\delta_j - \tilde{\alpha}_j (- \Delta)^{\frac{\gamma}{2}}\tilde{\delta}_j\right)(- \Delta)^{\frac{\gamma}{2}}\delta_idx & = & (\alpha_j - \tilde{\alpha}_j)\int_{\mathbb{R}^N}(- \Delta)^{\frac{\gamma}{2}}\delta_j(- \Delta)^{\frac{\gamma}{2}}\delta_idx\\
& + & \tilde{\alpha}_j\int_{\mathbb{R}^N}\delta_i^{\frac{N + 2\gamma}{N - 2\gamma}}(\delta_j - \tilde{\delta}_j)dx.
\end{eqnarray*}
Then we have
\begin{equation*}
\mu_i + \tilde{\alpha}_i\int_{\mathbb{R}^N}\delta_i^{\frac{N + 2\gamma}{N - 2\gamma}}(\delta_i - \tilde{\delta}_i)dx = o(1)(|\eta_j| + |a_j| + |\mu_j|) + o(1)\left(\int_{\mathbb{R}^N}|((- \Delta)^{\frac{\gamma}{2}} \delta_i - (- \Delta)^{\frac{\gamma}{2}} \tilde{\delta}_i)|^2dx\right)^{\frac{1}{2}},
\end{equation*}
and
\begin{equation*}
\int_{\mathbb{R}^N}\delta_i^{\frac{N + 2\gamma}{N - 2\gamma}}(\delta_i - \tilde{\delta}_i)dx = O(|a_i|^2 + |\eta_i|^2).
\end{equation*}
Hence,
\begin{equation*}
\mu_i = o(1)(|\eta_j| + |a_j| + |\mu_j|).
\end{equation*}
Similarly, it holds that
\begin{equation*}
\eta_i = o(1)(|\eta_j| + |a_j| + |\mu_j|)
\end{equation*}
and
\begin{equation*}
a_i = o(1)(|\eta_j| + |a_j| + |\mu_j|).
\end{equation*}
Thus
\begin{equation*}
\eta_i = 0, \quad a_i = 0, \quad \mu_i =0 \quad \text{ for all } i= 1, 2.
\end{equation*}
This completes the proof.
\end{proof}

By Lemma \ref{l:uniqueness}, we have the following direct consequence,
\begin{lemma}
There is a constant $\mu_0 > 0$ such that when $\mu\in (0, \mu_0]$ and each $u\in \dot{H}^\gamma(\mathbb{R}^N)$ satisfying
\begin{equation*}
\|u - \sum_{j = 1}^2 U_{\bar{y}^j, \bar{\lambda}_j}\|_{\dot{H}^\gamma(\mathbb{R}^N)}\leq \mu, \quad\text{for some } (\bar{y}, \bar{\lambda})\in D_\mu,
\end{equation*}
$u$ can be uniquely decomposed into
$$
u = \sum_{j = 1}^2\alpha_j U_{y^j, \lambda_j} + v,
$$
for some $(\alpha, y, \lambda, v)\in M_\mu$.
\end{lemma}
Thus we clearly have
\begin{lemma}\label{e:criticalpoint}
For $\mu > 0$ sufficiently small, $(\alpha, y, \lambda, v)\in M_\mu$ is a critical point of $J_\varepsilon$ in $M_\varepsilon$ if and only if $u = \sum_{j = 1}^2\alpha_j U_{y^j, \lambda_j} + v$ is a critical point of $I_\varepsilon$ in $\dot{H}^\gamma(\mathbb{R}^N)$.
\end{lemma}

Finally, we prove
\begin{lemma}\label{e:positive}
If $u_\varepsilon = \sum_{j = 1}^2\alpha_j U_{y^j, \lambda_j} + v$ is a critical point of (\ref{e:functional}), then $u$ is positive.
\end{lemma}
\begin{proof}
We follow the idea of \cite{ReyOlivier1990}.
Let us write $u_\varepsilon = u_\varepsilon^+ + u_\varepsilon^-$, where $u_\varepsilon^+ = \max(0, u_\varepsilon)$ and $u_\varepsilon^- = \min(0, u_\varepsilon)$. Suppose $u_\varepsilon^-\neq 0$. Testing equation (\ref{e:main}) by $u_\varepsilon^-$, we have
\begin{equation*}
\displaystyle\int\int_{\mathbb{R}^N\times\mathbb{R}^N}\frac{(u_\varepsilon(x) - u_\varepsilon(y))(u_\varepsilon^-(x) - u_\varepsilon^-(y))}{|x -y|^{N + 2\gamma}}dxdy = \int_{\mathbb{R}^N}(1 + \varepsilon K(x))|u_\varepsilon^-|^{2^*_\gamma}dx.
\end{equation*}
Since
\begin{equation*}
\displaystyle\int\int_{\mathbb{R}^N\times\mathbb{R}^N}\frac{(u_\varepsilon(x) - u_\varepsilon(y))(u_\varepsilon^-(x) - u_\varepsilon^-(y))}{|x -y|^{N + 2\gamma}}dxdy \geq \displaystyle\int\int_{\mathbb{R}^N\times\mathbb{R}^N}\frac{|u_\varepsilon^-(x) - u_\varepsilon^-(y)|^2}{|x -y|^{N + 2\gamma}}dxdy,
\end{equation*}
by the Sobolev embedding inequality, it holds that
\begin{equation*}
\int_{\mathbb{R}^N}|u_\varepsilon^-|^{2^*_\gamma}dx \geq c_0.
\end{equation*}
But
\begin{equation*}
\int_{\mathbb{R}^N}|u_\varepsilon^-|^{2^*_\gamma}dx \leq \int_{\mathbb{R}^N}|u_\varepsilon|^{2^*_\gamma}dx \to 0\quad \text{as}\,\,\,\, \varepsilon \to 0,
\end{equation*}
which is a contradiction. Hence $u_\varepsilon \geq 0$. By the maximum principle for the fractional Laplacian, see \cite{CabreSirepreprint}, $u_\varepsilon > 0$. This completes the proof.
\end{proof}

%% file: AppendixB.tex
\section{}
In this section, we prove the invertibility of $Q_\varepsilon$. We follow the idea of \cite{DancerYan1999} and \cite{DancerYan2007}. First, we prove that
\begin{lemma}\label{l:lemmaC1}
Suppose $(y, \lambda)\in D_\mu$, $\mu$ is small enough, the operator defined by
\begin{eqnarray*}
\langle A_{y, \lambda} v, w\rangle = \langle v, w\rangle -(2^*_\gamma - 1)\displaystyle\int_{\mathbb{R}^n}\sum_{j = 1}^2 U_{y^j, \lambda_{j}}^{2^*_\gamma - 2}vwdx,\quad v, w\in E^2_{y, \lambda}
\end{eqnarray*}
satisfies
\begin{eqnarray*}
\|A_{y, \lambda} v\|_{\dot{H}^\gamma(\mathbb{R}^N)} \geq c_0\|v\|_{\dot{H}^\gamma(\mathbb{R}^N)}\quad \text{ for }v\in E^2_{y, \lambda}.
\end{eqnarray*}.
\end{lemma}
\begin{proof}
We argue by contradiction. Suppose that there exists $\mu_k \to 0$ and $(y_k, \lambda^k)\in D_{\mu_k}$, $v_k\in E^2_{y_k, \lambda^k}$ such that
\begin{eqnarray*}
\|A_{y_k, \lambda^k} v_k\|_{\dot{H}^\gamma(\mathbb{R}^N)} = o(1)\|v_k\|_{\dot{H}^\gamma(\mathbb{R}^N)}.
\end{eqnarray*}

We may assume
\begin{equation*}
\|v_k\|_{\dot{H}^\gamma(\mathbb{R}^N)} = (\lambda_1^k\lambda_2^k)^{-\frac{N}{2}}.
\end{equation*}
So
\begin{eqnarray*}
\left|\langle A_{y_k, \lambda^k} v_k, w\rangle\right| = o((\lambda_1^k\lambda_2^k)^{-\frac{N}{2}})\|w\|_{\dot{H}^\gamma(\mathbb{R}^N)}.
\end{eqnarray*}
That is,
\begin{eqnarray*}
\displaystyle\int_{\mathbb{R}^n}(- \Delta)^{\frac{\gamma}{2}}v_k(- \Delta)^{\frac{\gamma}{2}}wdx &-&(2^*_\gamma - 1)\displaystyle\int_{\mathbb{R}^n}\sum_{j = 1}^2 U_{y_k^j, \lambda_{j}^k}^{2^*_\gamma - 2}v_kwdx\\ &=& o((\lambda_1^k\lambda_2^k)^{-\frac{N}{2}})\|w\|_{\dot{H}^\gamma(\mathbb{R}^N)},\quad w\in E^2_{y_k, \lambda^k}.
\end{eqnarray*}

For each fixed $i$, let
$$
\bar{v}_k(x) = v_k(\frac{1}{\lambda^k_i}x + y^i_{k}),
$$
then we have
\begin{eqnarray*}
\displaystyle\int_{\mathbb{R}^n}(- \Delta)^{\frac{\gamma}{2}}\bar{v}_k(- \Delta)^{\frac{\gamma}{2}}wdx &-&(2^*_\gamma - 1)\displaystyle\int_{\mathbb{R}^n}\frac{1}{(1 + |x|^2)^{2\gamma}}\bar{v}_kwdx\\ &-&(2^*_\gamma - 1)\displaystyle\int_{\mathbb{R}^n}\frac{(\lambda_j^k/\lambda_i^k)^{2\gamma}}{(1 + |\frac{\lambda_j^k}{\lambda_i^k}x + \lambda_j^k(y^i_k - y^j_k)|^2)^{2\gamma}}\bar{v}_kwdx\\ &=& (\lambda_i^k)^{- 2\gamma}(\lambda_j^k)^{-\frac{N}{2}}o(1)\|w\|_{\dot{H}^\gamma(\mathbb{R}^N)}, \quad w\in F^2_{y_k, \lambda^k}
\end{eqnarray*}
where
\begin{eqnarray*}
F^2_{y_k, \lambda^k} = \{v\in \dot{H}^\gamma(\mathbb{R}^N)\,\,\,\,|&&\displaystyle\langle U_{y^j_k, \lambda_j^k}(\frac{1}{\lambda^k_i}\cdot + y^i_{k}), v\rangle = \langle\frac{\partial U_{y^j_k, \lambda_j^k}}{\partial \lambda_j^k}(\frac{1}{\lambda^k_i}\cdot + y^i_{k}), v \rangle\\ &&= \langle\frac{\partial U_{y^j_k, \lambda_j^k}}{\partial y^j_{ki}}(\frac{1}{\lambda^k_i}\cdot + y^i_{k}), v\rangle = 0,\text{ for } j = 1,2,\quad i = 1, \cdot\cdot\cdot, N\}.
\end{eqnarray*}
Since $\|\bar{v}_k\|_{\dot{H}^\gamma(\mathbb{R}^N)} = (\lambda_i^k)^{- 2\gamma}(\lambda_j^k)^{-\frac{N}{2}}$, we may conclude that
\begin{equation*}
\bar{v}_k\rightharpoonup v \quad \text{ weakly in } \dot{H}^\gamma(\mathbb{R}^N),
\end{equation*}
\begin{equation*}
\bar{v}_k\rightarrow v \quad \text{ strongly in } L^p(\mathbb{R}^N) \text{ with } p\in [2, 2^*_\gamma).
\end{equation*}
Moreover, it is easy to see that $v$ satisfies
\begin{eqnarray*}
\displaystyle\langle U, v\rangle = \langle\frac{\partial U}{\partial x_j}, v \rangle  = 0.
\end{eqnarray*}

Now we claim that $v = 0$. Assume this for the moment. Since for each $L > 0$, we have
\begin{eqnarray*}
\displaystyle\sum_{j = 1}^2 \int_{\mathbb{R}^n}U_{y_k^j, \lambda_{j}^k}^{2^*_\gamma - 2}v_k^2dx & = &\sum_{j = 1}^2 \int_{B_{L\frac{1}{\lambda^k_i}}(y^i_{k})}U_{y_k^j, \lambda_{j}^k}^{2^*_\gamma - 2}v_k^2dx + (\frac{1}{L^{2N - 1}})^{\frac{4\gamma}{N + 2\gamma}}(\lambda_1^k\lambda_2^k)^{-N}\\
& = & \left((\frac{1}{\lambda^k_1})^{-N} + (\frac{1}{\lambda^k_1})^{-N}\right)o(1) + (\frac{1}{L^{2N - 1}})^{\frac{4\gamma}{N + 2\gamma}}(\lambda_1^k\lambda_2^k)^{-N}\\& = &(\lambda_1^k\lambda_2^k)^{-N}(o(1) + (\frac{1}{L^{2N - 1}})^{\frac{4\gamma}{N + 2\gamma}}).
\end{eqnarray*}
Then we have
\begin{equation*}
\|v_k\|^2_{\dot{H}^\gamma(\mathbb{R}^N)} = (\lambda_1^k\lambda_2^k)^{-N}(o(1) + (\frac{1}{L^{2N - 1}})^{\frac{4\gamma}{N + 2\gamma}}).
\end{equation*}
This is a contradiction.

So it  remains to prove that $v = 0$. First, we claim that $v$ satisfies
\begin{equation}\label{e:appendixC}
\displaystyle\int_{\mathbb{R}^n}(- \Delta)^{\frac{\gamma}{2}}v(- \Delta)^{\frac{\gamma}{2}}wdx -(2^*_\gamma - 1)\displaystyle\int_{\mathbb{R}^n}U^{2^*_\gamma - 2}vwdx = 0,\quad w\in F,
\end{equation}
where
\begin{eqnarray*}
F = \{w\in \dot{H}^\gamma(\mathbb{R}^N)\,\,|\,\,\displaystyle\langle U, w\rangle = \langle\frac{\partial U}{\partial x_j}, w\rangle = 0, \quad j = 1,2,\cdot\cdot\cdot, N\}.
\end{eqnarray*}
Indeed, for each $w\in F$, we can choose $\alpha_j^{(k)}$, $\beta_j^{(k)}$ and $\gamma_{jl}^{(k)}$ such that
\begin{eqnarray*}
\eta_k & = & w - \sum_{j = 1}^2\alpha_j^{(k)}U_{y^j_k, \lambda_j^k}(\frac{1}{\lambda^k_i}\cdot + y^i_{k}) - \sum_{j = 1}^2\beta_j^{(k)}\frac{\partial U_{y^j_k, \lambda_j^k}}{\partial \lambda_j^k}(\frac{1}{\lambda^k_i}\cdot + y^i_{k})\\ &-& \sum_{j = 1}^2\sum_{l = 1}^2\gamma_{jl}^{(k)}\frac{\partial U_{y^j_k, \lambda_j^k}}{\partial y^j_{kl}}(\frac{1}{\lambda^k_i}\cdot + y^i_{k})\in F^2_{y_k, \lambda^k}.
\end{eqnarray*}
And it is easy to see that $\alpha_j^{(k)} \to 0$, $\beta_j^{(k)}\to 0$ and $\gamma_{jl}^{(k)}\to 0$ as $k \to \infty$. Hence we have (\ref{e:appendixC}).
Then from the non-degenerate result of \cite{DavilaDelpinoSire2013}, $v = 0$. This completes the proof.
\end{proof}
\begin{lemma}\label{l:lemmaC2}
Suppose $(y, \lambda)\in D_\mu$, $\mu$ and $\varepsilon$ small enough, then there exists $\delta > 0$ such that for $v\in E^2_{y, \lambda}$ we have
\begin{eqnarray*}
\|v\|^2_{\dot{H}^\gamma(\mathbb{R}^N)} -(2^*_\gamma - 1)\displaystyle\int_{\mathbb{R}^n}\left(1 + \varepsilon K\right)\left(\sum_{j = 1}^2 \hat{\alpha}U_{y^j, \lambda_{j}}\right)^{2^*_\gamma - 2}v^2dx\geq \delta \|v\|^2_{\dot{H}^\gamma(\mathbb{R}^N)}.
\end{eqnarray*}
\end{lemma}
The proof of this lemma is the same as Lemma A.4 in \cite{CaoNoussairYan2002}, so we omit it.

%% file: AppendixC.tex
\section{}
In this appendix, we prove some estimates needed in the proof of our main results.
\begin{lemma}
Let $\alpha$, $\beta > 1$ such that $\alpha + \beta = 2^*_\gamma$, then there exists $\theta > 0$ such that
\begin{equation*}
\displaystyle\int_{\mathbb{R}^N}U_{y^i, \lambda_i}^{\frac{N + 2}{N - 2}}U_{y^j, \lambda_j}dx = C_0^{\frac{2N}{N - 2\gamma}}C_1 \varepsilon_{ij} + O(\varepsilon_{ij}^{\frac{N}{N - 2\gamma}}),
\end{equation*}
\begin{equation*}
\displaystyle\int_{\mathbb{R}^N}U_{y^i, \lambda_i}^{\frac{N}{N - 2\gamma}}U_{y^j, \lambda_j}^{\frac{N}{N - 2\gamma}}dx = O(\varepsilon_{ij}^{\frac{N}{N - 2\gamma}}\log \varepsilon_{ij}^{-1}),
\end{equation*}
\begin{equation*}
\displaystyle\int_{\mathbb{R}^N}U_{y^i, \lambda_i}^{\alpha}U_{y^j, \lambda_j}^{\beta}dx = O(\varepsilon_{ij}(\log \varepsilon_{ij}^{-1})^{\frac{N - 2\gamma}{N}\theta})\quad\text{ with } \theta = \inf(\alpha, \beta),
\end{equation*}
\begin{equation*}
|\frac{\partial \delta_i}{\partial x_i}|\leq C\lambda_i,
\end{equation*}
\begin{equation*}
|\frac{\partial \delta_i}{\partial \lambda_i}|\leq \frac{C}{\lambda_i}.
\end{equation*}
\end{lemma}
\begin{proof}
The proof is very similar to that of \cite{BahriACriticalPIVC1988} and \cite{ReyOlivier1990}, so we only prove the first estimate.
Set
\begin{equation*}
\delta_i(x) = C_0\left(\frac{\lambda_i}{1 + \lambda_i^2 |x - x_i|^2}\right)^{\frac{N - 2\gamma}{2}}.
\end{equation*}
\begin{equation*}
I = \displaystyle\int_{\mathbb{R}^N}\delta_i^{\frac{N + 2\gamma}{N - 2\gamma}}\delta_jdx.
\end{equation*}
Then it holds that
\begin{eqnarray*}
I &=& C_0^{\frac{2N}{N - 2\gamma}}\displaystyle\int_{\mathbb{R}^N}\frac{1}{(1 + |x|^2)^{\frac{N + 2\gamma}{2}}(\frac{\lambda_j}{\lambda_i} + |\sqrt{\frac{\lambda_i}{\lambda_j}} - \sqrt{\lambda_i\lambda_j}d_{ij}|^2)^{\frac{N - 2\gamma}{2}}}dx \\
&=& C_0^{\frac{2N}{N - 2\gamma}}\displaystyle\int_{\mathbb{R}^N}\frac{1}{(1 + |x|^2)^{\frac{N + 2\gamma}{2}}(\frac{\lambda_i}{\lambda_j} + |\sqrt{\frac{\lambda_j}{\lambda_i}} + \sqrt{\lambda_i\lambda_j}d_{ij}|^2)^{\frac{N - 2\gamma}{2}}}dx.
\end{eqnarray*}
First, we assume that
$$
\mu = \max\{\lambda_i/\lambda_j, \lambda_j/\lambda_i, \lambda_i\lambda_j|x_i - x_j|^2\} = \lambda_i/\lambda_j.
$$
By Taylor expansion, it holds that
\begin{equation*}
\frac{\lambda_i}{\lambda_j} + |\sqrt{\frac{\lambda_j}{\lambda_i}} + \sqrt{\lambda_i\lambda_j}d_{ij}|^2 = \left(\frac{\lambda_i}{\lambda_j} + \lambda_i\lambda_j|d_{ij}|^2\right) \times \left\{1 + \frac{\frac{\lambda_j}{\lambda_i}|x|^2 + 2\lambda_j x\cdot d_{ij}}{\frac{\lambda_i}{\lambda_j} + \lambda_i\lambda_j|d_{ij}|^2}\right\}.
\end{equation*}
\begin{eqnarray*}
&&\left\{\frac{1}{\frac{\lambda_i}{\lambda_j} + |\sqrt{\frac{\lambda_j}{\lambda_i}} + \sqrt{\lambda_i\lambda_j}d_{ij}|^2}\right\}^{\frac{N - 2\gamma}{2}} = \left(\frac{1}{\frac{\lambda_i}{\lambda_j} + \lambda_i\lambda_j|d_{ij}|^2}\right)^{\frac{N - 2\gamma}{2}}\\ && \times \left\{1 - (N - 2\gamma)\frac{\lambda_j x\cdot d_{ij}}{\frac{\lambda_i}{\lambda_j} + \lambda_i\lambda_j|d_{ij}|^2} +  O\left(\frac{1}{\lambda_i\lambda_j|d_{ij}|^2}|x|^2\right)\right\}.
\end{eqnarray*}
Then we have
\begin{eqnarray*}
&& \displaystyle\int_{|x|\leq \frac{\sqrt{\mu}}{10}}\frac{1}{(1 + |x|^2)^{\frac{N + 2\gamma}{2}}(\frac{\lambda_i}{\lambda_j} + |\sqrt{\frac{\lambda_j}{\lambda_i}} + \sqrt{\lambda_i\lambda_j}d_{ij}|^2)^{\frac{N - 2\gamma}{2}}}dx\\ &&= \left(\frac{1}{\frac{\lambda_i}{\lambda_j} + \lambda_i\lambda_j|d_{ij}|^2}\right)^{\frac{N - 2\gamma}{2}}\left(C_1 + O\left(\frac{1}{\mu^{\gamma}}\right)\right).
\end{eqnarray*}
Moreover, by easy computations, we have
\begin{equation*}
\displaystyle\int_{|x|\geq \frac{\sqrt{\mu}}{10}}\frac{1}{(1 + |x|^2)^{\frac{N + 2\gamma}{2}}}dx = O\left(\left(\frac{1}{\mu}\right)^{\gamma}\right),
\end{equation*}
\begin{equation*}
\displaystyle\int_{|x|\leq \frac{\sqrt{\mu}}{10}}\frac{x}{(1 + |x|^2)^{\frac{N + 2\gamma}{2}}}dx = 0,
\end{equation*}
\begin{equation*}
\displaystyle\int_{|x|\leq \frac{\sqrt{\mu}}{10}}\frac{1}{(1 + |x|^2)^{\frac{N + 2\gamma}{2}}}dx = C_1 + O\left(\left(\frac{1}{\mu}\right)^{\gamma}\right),
\end{equation*}
\begin{equation*}
\displaystyle\int_{|x|\leq \frac{\sqrt{\mu}}{10}}\frac{x^2}{(1 + |x|^2)^{\frac{N + 2\gamma}{2}}}dx = \mu^{1 -\gamma}.
\end{equation*}
Hence it holds that
\begin{eqnarray*}
&&\displaystyle\int_{|x|\leq \frac{\sqrt{\mu}}{10}}\frac{1}{(1 + |x|^2)^{\frac{N + 2\gamma}{2}}(\frac{\lambda_i}{\lambda_j} + |\sqrt{\frac{\lambda_j}{\lambda_i}} + \sqrt{\lambda_i\lambda_j}d_{ij}|^2)^{\frac{N - 2\gamma}{2}}}dx\\ && = \left(\frac{1}{\frac{\lambda_i}{\lambda_j} + \lambda_i\lambda_j|d_{ij}|^2}\right)^{\frac{N - 2\gamma}{2}}\left(C_1 + O\left(\frac{1}{\mu^{\gamma}}\right)\right)\\
&& = C_1 \varepsilon_{ij} + O(\varepsilon_{ij}^{\frac{N}{N - 2\gamma}}).
\end{eqnarray*}

For the case $\mu = \lambda_j/\lambda_i$, the proof is similar. So we are left with the case
$$
\mu = \lambda_i\lambda_j|x_i - x_j|^2.
$$

Let
\begin{equation*}
B_1 = \{x\in \mathbb{R}^N| |x + \lambda_i d_{ij}|\leq 1/10\lambda_i |d_{ij}|\},
\end{equation*}
\begin{equation*}
B_2 = \{x\in \mathbb{R}^N| |x|\leq \frac{\sqrt{\mu}}{10}\}
\end{equation*}
and
\begin{equation*}
L(x) = \frac{1}{(1 + |x|^2)^{\frac{N + 2\gamma}{2}}(\frac{\lambda_i}{\lambda_j} + |\sqrt{\frac{\lambda_j}{\lambda_i}}x + \sqrt{\lambda_i\lambda_j}d_{ij}|^2)^{\frac{N - 2\gamma}{2}}}.
\end{equation*}
Then
\begin{equation*}
\displaystyle\int_{(B_1 \cup B_2)^c}L(x)dx \leq \frac{C}{\mu^{\frac{N - 2\gamma}{2}}}\int_{\sqrt{\mu}}^{+\infty}\frac{r^{N - 1}dr}{(1 + r^2)^{\frac{N + 2\gamma}{2}}} = O(\varepsilon_{ij}^{\frac{N}{N - 2\gamma}}).
\end{equation*}

On $B_1$, we have
\begin{equation*}
|x|\geq 9/10 \lambda_i |d_{ij}|.
\end{equation*}
Thus
\begin{eqnarray*}
\displaystyle\int_{B_1}L(x)dx &&\leq \frac{C}{\lambda_i^{N + 2\gamma}|d_{ij}|^{N + 2\gamma}}\left(\frac{\lambda_j}{\lambda_i}\right)^{\frac{N - 2\gamma}{2}}\int_{0}^{\lambda_i|d_{ij}|}\frac{r^{N - 1}dr}{(1 + \frac{\lambda_j^2}{\lambda_i^2}r^2)^{\frac{N - 2\gamma}{2}}}\\
&&\leq \frac{C}{\lambda_i^{N + 2\gamma}|d_{ij}|^{N + 2\gamma}}\left(\frac{\lambda_i}{\lambda_j}\right)^{\frac{N + 2\gamma}{2}}\left(\lambda_i\lambda_j|d_{ij}|^2\right)^{\gamma} = O(\varepsilon_{ij}^{\frac{N}{N - 2\gamma}}).
\end{eqnarray*}
On the other hand, we have
\begin{equation*}
S = \displaystyle\int_{\mathbb{R}^N}\delta_i^{\frac{N}{N - 2\gamma}}\delta_j^{\frac{N}{N - 2\gamma}}dx = C_0^{\frac{2N}{N - 2\gamma}}\displaystyle\int_{\mathbb{R}^N}\frac{(\lambda_i\lambda_j)^{N/2}}{(1 + \lambda_i^2 |x - x_i|^2)^{\frac{N}{2}}(1 + \lambda_j^2 |x - x_j|^2)^{\frac{N}{2}}}dx
\end{equation*}
Let
\begin{equation*}
a_{ij} = \frac{x_j - x_i}{2}, \quad z = x- \frac{x_i + x_j}{2}.
\end{equation*}
Then we have
\begin{equation*}
S = C_0^{\frac{2N}{N - 2\gamma}}\frac{1}{(\lambda_i\lambda_j)^{N/2}}\displaystyle\int_{\mathbb{R}^N}\frac{1}{(\frac{1}{\lambda_i^2} + |z + a_{ij}|^2)^{\frac{N}{2}}(\frac{1}{\lambda_j^2} + |z - a_{ij}|^2)^{\frac{N}{2}}}dx.
\end{equation*}
Combining the above computations, we have the desired estimates. This completes the proof.
\end{proof}
\begin{lemma}\label{l:lemmaB2}
For any $(y, \lambda)\in D_\mu$ and $v\in E^2_{y, \lambda}$, we have, for some $\tau > 0$,
\begin{equation*}
\displaystyle\int_{\mathbb{R}^n}K\left(\sum_{j = 1}^2\alpha_{j} U_{y^j, \lambda_{j}}\right)^{2^*_\gamma - 1}vdx = O\left(\displaystyle\sum_{j = 1}^2\left(\frac{1}{\lambda_j^{\beta_j}} + |y^j - z^j|^{\beta_j}\right) + \varepsilon_{12}^{\frac{1}{2} + \tau}\right)\|v\|_{\dot{H}^\gamma(\mathbb{R}^N)}.
\end{equation*}
\end{lemma}
\begin{proof}
First, we have
\begin{eqnarray*}
\displaystyle\int_{\mathbb{R}^n}K(y)U_{x, \lambda}^{2^*_\gamma - 1}vdy & = &O\left(\int_{\mathbb{R}^n}\left(|x - z_j|^{\beta_j} + \frac{|y|^{\beta_j}}{\lambda^{\beta_j}}\right)U^{2^*_\gamma - 1}|v(\frac{y}{\lambda} + x)|\lambda^{\frac{2\gamma - n}{2}}dy\right)\\
& = & O\left(|x - z_j|^{\beta_j} + \frac{1}{\lambda^{\beta_j}}\right)\|v\|_{\dot{H}^\gamma(\mathbb{R}^N)}.
\end{eqnarray*}
Then
\begin{eqnarray*}
&&\displaystyle\int_{\mathbb{R}^n}K\left(\sum_{j = 1}^2\alpha_{j} U_{y^j, \lambda_{j}}\right)^{2^*_\gamma - 1}vdx\\ && = \displaystyle\sum_{j = 1}^2\int_{\mathbb{R}^n}K\alpha_{j}^{2^*_\gamma - 1} U_{y^j, \lambda_{j}}^{2^*_\gamma - 1}vdx + O\left(\int_{\mathbb{R}^n}K(y)U_{y^1, \lambda_{1}}^{2^*_\gamma - 1}U_{y^2, \lambda_{2}}^{2^*_\gamma - 1}|v|dx\right)\\
&& = \displaystyle\sum_{j = 1}^2\int_{\mathbb{R}^n}K\alpha_{j}^{2^*_\gamma - 1} U_{y^j, \lambda_{j}}^{2^*_\gamma - 1}vdx + O\left(\varepsilon_{12}^{\frac{1}{2} + \tau}\right)\|v\|_{\dot{H}^\gamma(\mathbb{R}^N)}.
\end{eqnarray*}
Finally,
\begin{equation*}
\displaystyle\int_{\mathbb{R}^n}K\left(\sum_{j = 1}^2\alpha_{j} U_{y^j, \lambda_{j}}\right)^{2^*_\gamma - 1}vdx = \displaystyle\sum_{j = 1}^2\int_{\mathbb{R}^n}K\alpha_{j}^{2^*_\gamma - 1} U_{y^j, \lambda_{j}}^{2^*_\gamma - 1}vdx + O\left(\varepsilon_{12}\right)\|v\|_{\dot{H}^\gamma(\mathbb{R}^N)}.
\end{equation*}
\end{proof}
\begin{lemma}\label{l:lemmaB3}
For any $(y, \lambda)\in D_\mu$ and $v\in E^2_{y, \lambda}$, when $\mu$ is small, we have, for some $\tau > 0$,
\begin{equation*}
\displaystyle\int_{\mathbb{R}^n}\left(\sum_{j = 1}^2\alpha_{j} U_{y^j, \lambda_{j}}\right)^{2^*_\gamma - 1}vdx = O\left(\varepsilon_{12}^{\frac{1}{2} + \tau}\right)\|v\|_{\dot{H}^\gamma(\mathbb{R}^N)}.
\end{equation*}
\end{lemma}
\begin{proof}
\begin{eqnarray*}
&&\displaystyle\int_{\mathbb{R}^n}\left(\sum_{j = 1}^2\alpha_{j} U_{y^j, \lambda_{j}}\right)^{2^*_\gamma - 1}vdx\\ && = \displaystyle\sum_{j = 1}^2\int_{\mathbb{R}^n}\alpha_{j}^{2^*_\gamma - 1} U_{y^j, \lambda_{j}}^{2^*_\gamma - 1}vdx + O\left(\int_{\mathbb{R}^n}U_{y^1, \lambda_{1}}^{2^*_\gamma - 1}U_{y^2, \lambda_{2}}^{2^*_\gamma - 1}|v|dx\right)\\
&& = \displaystyle\sum_{j = 1}^2\int_{\mathbb{R}^n}\alpha_{j}^{2^*_\gamma - 1} U_{y^j, \lambda_{j}}^{2^*_\gamma - 1}vdx + O\left(\varepsilon_{12}^{\frac{1}{2} + \tau}\right)\|v\|_{\dot{H}^\gamma(\mathbb{R}^N)}\\
&&
= O\left(\varepsilon_{12}^{\frac{1}{2} + \tau}\right)\|v\|_{\dot{H}^\gamma(\mathbb{R}^N)}.
\end{eqnarray*}
\begin{eqnarray*}
&&\displaystyle\int_{\mathbb{R}^n}\left(\sum_{j = 1}^2\alpha_{j} U_{y^j, \lambda_{j}}\right)^{2^*_\gamma - 1}vdx\\ &&= \displaystyle\sum_{j = 1}^2\int_{\mathbb{R}^n}\alpha_{j}^{2^*_\gamma - 1} U_{y^j, \lambda_{j}}^{2^*_\gamma - 1}vdx + O\left(\varepsilon_{12}\right)\|v\|_{\dot{H}^\gamma(\mathbb{R}^N)} \\ &&= O\left(\varepsilon_{12}\right)\|v\|_{\dot{H}^\gamma(\mathbb{R}^N)}.
\end{eqnarray*}
\end{proof}
\begin{lemma}\label{l:lemmaB4}
Suppose $(y, \lambda)\in D_\mu$, if $\mu$ is small, then we have
\begin{eqnarray*}
&&\left\langle \sum_{j = 1}^2 \hat{\alpha}_j U_{y^j, \lambda_j}, U_{y^k, \lambda_k}\right\rangle - \int_{\mathbb{R}^N}(1 + \varepsilon K)\left(\sum_{j = 1}^2 \hat{\alpha}_j U_{y^j, \lambda_j}\right)^{2^*_\gamma - 1}U_{y^k, \lambda_k}\\
&& = O\left(\varepsilon\left(\frac{1}{\lambda_j^{\beta_k}} + |y^j - z^j|^{\beta_k}\right) + \varepsilon_{12}^{\frac{1}{2} + \tau}\right).
\end{eqnarray*}
\end{lemma}
\begin{proof}
\begin{eqnarray*}
&&\left\langle \sum_{j = 1}^2 \hat{\alpha}_j U_{y^j, \lambda_j}, U_{y^k, \lambda_k}\right\rangle - \int_{\mathbb{R}^N}(1 + \varepsilon K)\left(\sum_{j = 1}^2 \hat{\alpha}_j U_{y^j, \lambda_j}\right)^{2^*_\gamma - 1}U_{y^k, \lambda_k}\\
&& = \hat{\alpha}_k\left(\int_{\mathbb{R}^N}U^{2^*_\gamma}_{y^k, \lambda_k} - \int_{\mathbb{R}^N}\frac{1 + \varepsilon K}{1 + \varepsilon K(z^k)}U^{2^*_\gamma}_{y^k, \lambda_k}\right) + O\left(\int_{\mathbb{R}^N}U_{y^1, \lambda_1}^{\alpha}U_{y^2, \lambda_2}^{\beta}\right)
\end{eqnarray*}
\begin{eqnarray*}
&& = \frac{1}{(1 + \varepsilon K)^{\frac{N + 2\gamma}{4\gamma}}}\int_{\mathbb{R}^N}\varepsilon(K(z^k) - K(x))U^{2^*_\gamma}_{y^k, \lambda_k} + O(\varepsilon_{12}^{\frac{1}{2}+\tau})\\
&& = O(\varepsilon)\int_{\mathbb{R}^N}|(K(z^k) - K(\frac{y}{\lambda_k} + y_k))|\frac{1}{( 1 + |y|^2)^{N}} + O(\varepsilon_{12}^{\frac{1}{2}+\tau})\\
&& = O\left(\varepsilon\left(\frac{1}{\lambda_j^{\beta_k}} + |y^j - z^j|^{\beta_k}\right) + \varepsilon_{12}^{\frac{1}{2} + \tau}\right).
\end{eqnarray*}
\end{proof}
\begin{lemma}
Suppose $(y, \lambda)\in D_\mu$, if $\mu$ and $\varepsilon$ are small, then we have
\begin{eqnarray*}
&&\langle U_{y^l, \lambda_l}, U_{y^k, \lambda_k}\rangle - (2^*_\gamma - 1)\int_{\mathbb{R}^N}(1 + \varepsilon K)\left(\displaystyle\sum_{j = 1}^2 \alpha_j U_{y^j, \lambda_j}\right)^{2^*_\gamma - 2} U_{y^l, \lambda_l}, U_{y^k, \lambda_k}\\
&& = \begin{cases} (1 - (2^*_\gamma - 1)\alpha_k^{2^*_\gamma - 2})A + O(\varepsilon_{12}^r) + O(\varepsilon), \quad\text{ if } k = l = 1, 2,\\
O(\varepsilon_{12}^r) + O(\varepsilon), \quad \text{ if } k\neq l, k, l = 1, 2
\end{cases}
\end{eqnarray*}
for some $r > 0$, $A > 0$.
\end{lemma}
\begin{proof}
\begin{eqnarray*}
&&\langle U_{y^l, \lambda_l}, U_{y^k, \lambda_k}\rangle - (2^*_\gamma - 1)\int_{\mathbb{R}^N}(1 + \varepsilon K)\left(\displaystyle\sum_{j = 1}^2 \alpha_j U_{y^j, \lambda_j}\right)^{2^*_\gamma - 2} U_{y^l, \lambda_l}, U_{y^k, \lambda_k}\\
&& = \int_{\mathbb{R}^N}U_{y^l, \lambda_l}^{2^*_\gamma - 1}U_{y^k, \lambda_k} - (2^*_\gamma - 1)\alpha_l^{2^*_\gamma - 2}\int_{\mathbb{R}^N}U_{y^l, \lambda_l}^{2^*_\gamma - 1}U_{y^k, \lambda_k}\\
&&\,\,\,\,\,\, +\,\, O\left(\int_{\mathbb{R}^N}\left(U_{y^1, \lambda_1}^{2^*_\gamma - 2}U_{y^2, \lambda_2}^2 + U_{y^1, \lambda_1}^{2^*_\gamma - 1}U_{y^2, \lambda_2} + U_{y^2, \lambda_2}^{2^*_\gamma - 2}U_{y^1, \lambda_1}^2 + U_{y^2, \lambda_2}^{2^*_\gamma - 1}U_{y^1, \lambda_1}\right)\right) + O(\varepsilon)\\
&&= (1  - (2^*_\gamma - 1)\alpha_l^{2^*_\gamma - 2})\int_{\mathbb{R}^N}U_{y^l, \lambda_l}^{2^*_\gamma - 1}U_{y^k, \lambda_k} + O\left(\varepsilon_{12}^r\right) + O(\varepsilon).
\end{eqnarray*}
Let $A = \int_{\mathbb{R}^N}U_{y^k, \lambda_k}^{2^*_\gamma} = \int_{\mathbb{R}^N}U_{0, 1}^{2^*_\gamma}$. We get the conclusion.
\end{proof}
\begin{lemma}
For any $(y, \lambda)\in D_\mu$ and $v\in E^2_{y, \lambda}$, if $\mu$ and $\varepsilon$ are small, then
\begin{eqnarray*}
&&\displaystyle\int_{\mathbb{R}^n}(1 + \varepsilon K)\left(\sum_{j = 1}^2\alpha_{j} U_{y^j, \lambda_{j}}\right)^{2^*_\gamma - 2}U_{y^k, \lambda_{k}}vdx\\ &&\,\,\,\,\,\,= O\left(\varepsilon\displaystyle\sum_{j = 1}^2\left(\frac{1}{\lambda_j^{\beta_j}} + |y^j - z^j|^{\beta_j}\right) + \varepsilon_{12}^{\frac{1}{2} + \tau}\right)\|v\|_{\dot{H}^\gamma(\mathbb{R}^N)}.
\end{eqnarray*}
\end{lemma}
\begin{proof}
\begin{eqnarray*}
&&\displaystyle\int_{\mathbb{R}^n}(1 + \varepsilon K)\left(\sum_{j = 1}^2\alpha_{j} U_{y^j, \lambda_{j}}\right)^{2^*_\gamma - 2}U_{y^k, \lambda_{k}}v dx \\
&&=\displaystyle\alpha_k^{2^*_\gamma - 2}\int_{\mathbb{R}^n}U_{y^k, \lambda_k}^{2^*_\gamma - 1}vdx + \displaystyle\alpha_k^{2^*_\gamma - 2}\varepsilon\int_{\mathbb{R}^n}K(x)U_{y^k, \lambda_k}^{2^*_\gamma - 1}v dx   \\
&&\,\,\,\,\,\,+\,\,\displaystyle\int_{\mathbb{R}^n}(1 + \varepsilon K)\left(\left(\sum_{j = 1}^2\alpha_{j} U_{y^j, \lambda_{j}}\right)^{2^*_\gamma - 2} - (\alpha_k U_{y^k, \lambda_k})^{2^*_\gamma - 2}\right)U_{y^k, \lambda_{k}}vdx\\
&& = \varepsilon O\left(\int_{\mathbb{R}^n}|x - z^k|^{\beta_k}U_{y^k, \lambda_k}^{2^*_\gamma - 1}|v|dx\right) + O\left(\displaystyle\int_{\mathbb{R}^n}U_{y^1, \lambda_1}^{\alpha}U_{y^2, \lambda_2}^{\beta}|v|dx\right)\\
&& = O\left(\varepsilon\displaystyle\sum_{j = 1}^2\left(\frac{1}{\lambda_j^{\beta_j}} + |y^j - z^j|^{\beta_j}\right) + \varepsilon_{12}^{\frac{1}{2} + \tau}\right)\|v\|_{\dot{H}^\gamma(\mathbb{R}^N)}.
\end{eqnarray*}
\end{proof}
Similarly, we have
\begin{lemma}\label{l:lemmaB7}
For any $(y, \lambda)\in D_\mu$ and $v\in E^2_{y, \lambda}$, if $\mu$ and $\varepsilon$ are small, then
\begin{eqnarray*}
&&\displaystyle\int_{\mathbb{R}^n}(1 + \varepsilon K)\left(\sum_{j = 1}^2\alpha_{j} U_{y^j, \lambda_{j}}\right)^{2^*_\gamma - 2}\frac{\partial U_{y^k, \lambda_{k}}}{\partial \lambda_k}vdx\\ &&= O\left(\frac{\varepsilon}{\lambda_k}\displaystyle\sum_{j = 1}^2\left(\frac{1}{\lambda_j^{\beta_j}} + |y^j - z^j|^{\beta_j}\right) + \frac{\varepsilon_{12}^{\frac{1}{2} + \tau}}{\lambda_k}\right)\|v\|_{\dot{H}^\gamma(\mathbb{R}^N)}.
\end{eqnarray*}
and
\begin{eqnarray*}
&&\displaystyle\int_{\mathbb{R}^n}(1 + \varepsilon K)\left(\sum_{j = 1}^2\alpha_{j} U_{y^j, \lambda_{j}}\right)^{2^*_\gamma - 2}\frac{\partial U_{y^k, \lambda_{k}}}{\partial y^k_i}vdx\\ && = O\left(\lambda_k\varepsilon\displaystyle\sum_{j = 1}^2\left(\frac{1}{\lambda_j^{\beta_j}} + |y^j - z^j|^{\beta_j}\right) + \lambda_k\varepsilon_{12}^{\frac{1}{2} + \tau}\right)\|v\|_{\dot{H}^\gamma(\mathbb{R}^N)}.
\end{eqnarray*}
\end{lemma}
\begin{lemma}\label{l:lemmaB8}
For any $(y, \lambda)\in D_\mu$, if $\mu$ is small, then
\begin{eqnarray*}
\displaystyle\int_{\mathbb{R}^N}KU_{y^k, \lambda_{k}}^{2^*_\gamma - 1}\frac{\partial U_{y^k, \lambda_{k}}}{\partial \lambda_k}dx & = &C_{N, \beta_k}\frac{1}{\lambda_k^{\beta_k + 1}}\displaystyle\sum_{i = 1}^Na_i^k + O\left(\frac{1}{\lambda_k^{\beta_k}}|y^k - z^k|\right)\\
& + & O\left(\frac{1}{\lambda_k^{\beta_k + 1 + \sigma}}\right)  + O\left(\frac{1}{\lambda_k}|y^k - z^k|^{\beta_k + \sigma}\right),
\end{eqnarray*}
where $C_{N, \beta_k}$ is a positive constant depending on $N$ and $\beta_k$ only.
\end{lemma}
\begin{proof}
\begin{eqnarray*}
&&\displaystyle\int_{\mathbb{R}^n}KU_{y^k, \lambda_{k}}^{2^*_\gamma - 1}\frac{\partial U_{y^k, \lambda_{k}}}{\partial \lambda_k}dx \\
&& = \displaystyle\int_{|x - z^k|\leq r_0}\left(\sum_{i = 1}^N a_i^k |x_i - z_i^k|^{\beta_k} + O(|x - z^k|^{\beta_k + \sigma})\right)U_{y^k, \lambda_{k}}^{2^*_\gamma - 1}\frac{\partial U_{y^k, \lambda_{k}}}{\partial \lambda_k}dx + O\left(\frac{1}{\lambda_k^{2N}}\right)\\
&& = \frac{C_0^{2^*_\gamma}}{\lambda_k^{\beta_{k} + 1}}\displaystyle\int_{\mathbb{R}^N}\frac{N - 2\gamma}{2}\sum_{i = 1}^N a_i^k |x_i + \lambda_k(y^k_i - z_i^k)|^{\beta_k}\frac{1 - |x|^2}{(1 + |x|^2)^{N + 1}}dx\\
&&\,\,\,\,\,\, +\,\, O\left(\frac{1}{\lambda_k^{\beta_k + 1 + \sigma}}\right) + O\left(\frac{1}{\lambda_k}|y^k - z^k|^{\beta_k + \sigma}\right)\\
&& = \frac{C_0^{2^*_\gamma}}{\lambda_k^{\beta_{k} + 1}}\frac{N - 2\gamma}{2}\displaystyle\int_{\mathbb{R}^N}\sum_{i = 1}^N a_i^k |x_i|^{\beta_k}\frac{1 - |x|^2}{(1 + |x|^2)^{N + 1}}dx\\ &&
\,\,\,\,\,\,+\,\, O\left(\frac{1}{\lambda_k^{\beta_k}}|y^k - z^k|\right) + O\left(\frac{1}{\lambda_k^{\beta_k + 1 + \sigma}}\right) + O\left(\frac{1}{\lambda_k}|y^k - z^k|^{\beta_k + \sigma}\right)
\end{eqnarray*}
\begin{eqnarray*}
&& = \frac{C_0^{2^*_\gamma}}{\lambda_k^{\beta_{k} + 1}}\frac{N - 2\gamma}{2N}\displaystyle\sum_{i = 1}^Na_i^k\int_{\mathbb{R}^N}|x|^{\beta_k}\frac{1 - |x|^2}{(1 + |x|^2)^{N + 1}}dx\\ &&
\,\,\,\,\,\,+\,\, O\left(\frac{1}{\lambda_k^{\beta_k}}|y^k - z^k|\right) + O\left(\frac{1}{\lambda_k^{\beta_k + 1 + \sigma}}\right) + O\left(\frac{1}{\lambda_k}|y^k - z^k|^{\beta_k + \sigma}\right)\\
&& = \frac{C_0^{2^*_\gamma}}{\lambda_k^{\beta_{k} + 1}}\frac{N - 2\gamma}{2N}\displaystyle\sum_{i = 1}^Na_i^k\frac{2n - 2\beta_k - 4}{2n - \beta_k -3}\int_{\mathbb{R}^N}|x|^{\beta_k}\frac{1}{(1 + |x|^2)^{N + 1}}dx\\
&&\,\,\,\,\,\,+\,\, O\left(\frac{1}{\lambda_k^{\beta_k}}|y^k - z^k|\right) + O\left(\frac{1}{\lambda_k^{\beta_k + 1 + \sigma}}\right) + O\left(\frac{1}{\lambda_k}|y^k - z^k|^{\beta_k + \sigma}\right)\\
&& = C_0^{2^*_\gamma}\frac{N - 2\gamma}{2N}\frac{2n - 2\beta_k - 4}{2n - \beta_k -3}\frac{1}{\lambda_k^{\beta_{k} + 1}}\int_{\mathbb{R}^N}|x|^{\beta_k}\frac{1}{(1 + |x|^2)^{N + 1}}dx\displaystyle\sum_{i = 1}^Na_i^k\\
&&\,\,\,\,\,\,+\,\, O\left(\frac{1}{\lambda_k^{\beta_k}}|y^k - z^k|\right) + O\left(\frac{1}{\lambda_k^{\beta_k + 1 + \sigma}}\right) + O\left(\frac{1}{\lambda_k}|y^k - z^k|^{\beta_k + \sigma}\right),
\end{eqnarray*}
\end{proof}
\begin{lemma}\label{l:lemmaB9}
For any $(y, \lambda)\in D_\mu$, if $\mu$ is small, then
\begin{eqnarray*}
\displaystyle\int_{\mathbb{R}^N}KU_{y^k, \lambda_{k}}^{2^*_\gamma - 1}\frac{\partial U_{y^k, \lambda_{k}}}{\partial y^k_i}dx & = &D_{N, \beta_k}\frac{1}{\lambda_k^{\beta_k - 1}}a_i^k (y^k_i - z^k_i)+ O\left(\frac{1}{\lambda_k^{\beta_k - 1}}\lambda_k^2|y^k - z^k|^2\right)
\end{eqnarray*}
\begin{eqnarray*}
&& +\,\, O\left(\frac{1}{\lambda_k^{\beta_k - 1 + \sigma}}\right)  + O\left(\lambda_k|y^k - z^k|^{\beta_k + \sigma}\right),
\end{eqnarray*}
where $D_{N, \beta_k}$ is a positive constant depending on $N$ and $\beta_k$ only.
\end{lemma}
\begin{proof}
\begin{eqnarray*}
&&\displaystyle\int_{\mathbb{R}^N}KU_{y^k, \lambda_{k}}^{2^*_\gamma - 1}\frac{\partial U_{y^k, \lambda_{k}}}{\partial y^k_i}dx = (N -2\gamma)\displaystyle\int_{\mathbb{R}^N}KU_{y^k, \lambda_{k}}^{2^*_\gamma}\frac{\lambda_k^2(x_i - y^k_i)}{1 + \lambda_k^2|x - y^k|^2}dx\\
&&= (N - 2\gamma)\displaystyle\int_{|x - z^k|\leq r_0}\left(\sum_{h = 1}^N a_h^k |x_h - z_h^k|^{\beta_k} + O(|x - z^k|^{\beta_k + \sigma})\right)U_{y^k, \lambda_{k}}^{2^*_\gamma}\frac{\lambda_k^2(x_i - y^k_i)}{1 + \lambda_k^2|x - y^k|^2}dx\\
&&\,\,\,\,\,\,+\,\, O\left(\frac{1}{\lambda_k^{N - 1}}\right)\\
&&= (N - 2\gamma)\lambda_k\displaystyle\int_{\mathbb{R}^N}\sum_{h = 1}^N a_h^k |x_h - z_h^k|^{\beta_k}U_{y^k, \lambda_{k}}^{2^*_\gamma}\frac{\lambda_k(x_i - y^k_i)}{1 + \lambda_k^2|x - y^k|^2}dx\\
&&\,\,\,\,\,\,+\,\, O\left(\frac{1}{\lambda_k^{\beta_k - 1 + \sigma}} + \lambda_k |y^k - z^k|^{\beta_k + \sigma}\right)
\end{eqnarray*}
\begin{eqnarray*}
&&= (N - 2\gamma)\frac{1}{\lambda_k^{\beta_k - 1}}\displaystyle\int_{\mathbb{R}^N}\sum_{h = 1}^N a_h^k (|x_h|^{\beta_k} + \beta_k |x_h|^{\beta_k - 2}x_h\lambda_k(y^k_h - z^k_h))U_{0,1}^{2^*_\gamma}\frac{x_i}{1 + |x|^2}dx\\
&&\,\,\,\,\,\,+\,\, O\left(\frac{1}{\lambda_k^{\beta_k - 1}}\lambda_k^2|y^k - z^k|^2\right) + O\left(\frac{1}{\lambda_k^{\beta_k - 1 + \sigma}} + \lambda_k |y^k - z^k|^{\beta_k + \sigma}\right)\\
&&= C_0^{2^*_\gamma}\frac{(N - 2\gamma)}{N}\frac{\beta_k}{\lambda_k^{\beta_k - 1}}a_i^k\displaystyle\int_{\mathbb{R}^N}\frac{|x|^{\beta_k}}{(1 + |x|^2)^{N + 1}}dx (y^k_h - z^k_h)\\
&&\,\,\,\,\,\,+\,\, O\left(\frac{1}{\lambda_k^{\beta_k - 1}}\lambda_k^2|y^k - z^k|^2\right) + O\left(\frac{1}{\lambda_k^{\beta_k - 1 + \sigma}} + \lambda_k |y^k - z^k|^{\beta_k + \sigma}\right).
\end{eqnarray*}
\end{proof}
Similarly, we have
\begin{lemma}\label{l:lemmaB10}
For any $(y, \lambda)\in D_\mu$, if $\mu$ is small and $k \neq l$, then
\begin{eqnarray*}
\displaystyle\int_{\mathbb{R}^N}U_{y^k, \lambda_{k}}^{2^*_\gamma - 2}\frac{\partial U_{y^k, \lambda_{k}}}{\partial \lambda_k}U_{y^l, \lambda_l}dx = \frac{1}{2^*_\gamma - 1}C_0\frac{\varepsilon_{12}}{\lambda_k |z^1 - z^2|^{N - 2\gamma}} + O\left(\frac{\varepsilon_{12}^{\frac{N}{N - 2\gamma}}}{\lambda_k}\right).
\end{eqnarray*}
\end{lemma}
and
\begin{lemma}\label{l:lemmaB11}
For any $(y, \lambda)\in D_\mu$, if $\mu$ is small and $k \neq l$, then
\begin{eqnarray*}
\displaystyle\int_{\mathbb{R}^N}U_{y^k, \lambda_{k}}^{2^*_\gamma - 2}\frac{\partial U_{y^k, \lambda_{k}}}{\partial y^k_i}U_{y^l, \lambda_l}dx = C_1\lambda_1\lambda_2(y^k_i - y^l_i)\varepsilon_{12}^{\frac{N}{N - 2\gamma}} + O\left(\varepsilon_{12}^{\frac{N - 1}{N - 2\gamma}}\right).
\end{eqnarray*}
\end{lemma}

%% file: AppendixD.tex
\section{}
The integration by parts formula for the fractional Laplacian is frequently used in this paper. We give its proof here.
\begin{lemma}
If $f$ and $g$ belongs to the space $\dot{H}^{2s}(\mathbb{R}^N)\cap L^2(\mathbb{R}^N)$, then it holds that
\begin{equation*}
\displaystyle\int_{\mathbb{R}^N}(- \Delta)^s f(x) g(x)dx = \int_{\mathbb{R}^N}(- \Delta)^s g(x) f(x)dx
\end{equation*}
\end{lemma}
\begin{proof}
\begin{eqnarray*}
\displaystyle\int_{\mathbb{R}^N}(- \Delta)^s f(x) g(x)dx &=& \int_{\mathbb{R}^N}g(x)\int_{\mathbb{R}^N}\frac{f(x) - f(y)}{|x -y|^{N + 2s}}dydx\\
& = & \int_{\mathbb{R}^N}\int_{\mathbb{R}^N}\frac{g(x)(f(x) - f(y))}{|x -y|^{N + 2s}}dydx\\
& = & \int_{\mathbb{R}^N}\int_{\mathbb{R}^N}\frac{g(y)(f(y) - f(x))}{|x -y|^{N + 2s}}dydx\\
& = & - \int_{\mathbb{R}^N}\int_{\mathbb{R}^N}\frac{g(y)(f(x) - f(y))}{|x -y|^{N + 2s}}dydx.
\end{eqnarray*}
Then
\begin{eqnarray*}
\displaystyle\int_{\mathbb{R}^N}(- \Delta)^s f(x) g(x)dx & = & \frac{1}{2}\int_{\mathbb{R}^N}\int_{\mathbb{R}^N}\frac{(g(x) - g(y))(f(x) - f(y))}{|x -y|^{N + 2s}}dydx\\
& = & \frac{1}{2}\int_{\mathbb{R}^N}(- \Delta)^s g(x) f(x)dx - \frac{1}{2}\int_{\mathbb{R}^N}\int_{\mathbb{R}^N}\frac{(g(x) - g(y))f(y)}{|x -y|^{N + 2s}}dydx\\
& = & \frac{1}{2}\int_{\mathbb{R}^N}(- \Delta)^s g(x) f(x)dx + \frac{1}{2}\int_{\mathbb{R}^N}\int_{\mathbb{R}^N}\frac{(g(x) - g(y))f(x)}{|x -y|^{N + 2s}}dydx\\
& = & \int_{\mathbb{R}^N}(- \Delta)^s g(x) f(x)dx.
\end{eqnarray*}
\end{proof}

%% file: main.bbl
\begin{thebibliography}{10}

\bibitem{AbdelhediChtiouiJFA2013}
Wael Abdelhedi and Hichem Chtioui.
\newblock On a {N}irenberg-type problem involving the square root of the
  {L}aplacian.
\newblock {\em J. Funct. Anal.}, 265(11):2937--2955, 2013.

\bibitem{AbdelouhabBonaFellandSaut1989}
L.~Abdelouhab, J.~L. Bona, M.~Felland, and J.-C. Saut.
\newblock Nonlocal models for nonlinear, dispersive waves.
\newblock {\em Phys. D}, 40(3):360--392, 1989.

\bibitem{AmbrosettiGarciaPeralJFA1999}
A.~Ambrosetti, J.~Garcia~Azorero, and I.~Peral.
\newblock Perturbation of {$\Delta u+u^{(N+2)/(N-2)}=0$}, the scalar curvature
  problem in {${\bf R}^N$}, and related topics.
\newblock {\em J. Funct. Anal.}, 165(1):117--149, 1999.

\bibitem{BahriACriticalPIVC1988}
A.~Bahri.
\newblock Critical points at infinity in the variational calculus.
\newblock In {\em Partial differential equations ({R}io de {J}aneiro, 1986)},
  volume 1324 of {\em Lecture Notes in Math.}, pages 1--29. Springer, Berlin,
  1988.

\bibitem{BahriCoron1988}
A.~Bahri and J.-M. Coron.
\newblock On a nonlinear elliptic equation involving the critical {S}obolev
  exponent: the effect of the topology of the domain.
\newblock {\em Comm. Pure Appl. Math.}, 41(3):253--294, 1988.

\bibitem{CabreSirepreprint}
Xavier Cabr{\'e} and Yannick Sire.
\newblock Nonlinear equations for fractional {L}aplacians, {I}: {R}egularity,
  maximum principles, and {H}amiltonian estimates.
\newblock {\em Ann. Inst. H. Poincar\'e Anal. Non Lin\'eaire}, 31(1):23--53,
  2014.

\bibitem{Caffarelli&Silvestre07}
Luis Caffarelli and Luis Silvestre.
\newblock An extension problem related to the fractional {L}aplacian.
\newblock {\em Comm. Partial Differential Equations}, 32(7-9):1245--1260, 2007.

\bibitem{CaoNoussairYan2002}
Daomin Cao, Ezzat~S. Noussair, and Shusen Yan.
\newblock On the scalar curvature equation {$-\Delta u=(1+\epsilon
  K)u^{(N+2)/(N-2)}$} in {$\Bbb R^N$}.
\newblock {\em Calc. Var. Partial Differential Equations}, 15(3):403--419,
  2002.

\bibitem{CaoPeng2013}
Daomin Cao and Shuangjie Peng.
\newblock Concentration of solutions for the {Y}amabe problem on half-spaces.
\newblock {\em Proc. Roy. Soc. Edinburgh Sect. A}, 143(1):73--99, 2013.

\bibitem{CaoPengYanWSCPCR2013}
Daomin Cao, Shuangjie Peng, and Shusen Yan.
\newblock On the {W}ebster scalar curvature problem on the {CR} sphere with a
  cylindrical-type symmetry.
\newblock {\em J. Geom. Anal.}, 23(4):1674--1702, 2013.

\bibitem{ChangGonzalez2011}
Sun-Yung~Alice Chang and Mar{\'{\i}}a del~Mar Gonz{\'a}lez.
\newblock Fractional {L}aplacian in conformal geometry.
\newblock {\em Adv. Math.}, 226(2):1410--1432, 2011.

\bibitem{ChenZheng2014}
Guoyuan Chen and Youquan Zheng.
\newblock A perturbation result for the ${Q}_\gamma$ curvature problem on
  ${S}^n$.
\newblock {\em Nonlinear Analysis: Theory, Methods and Applications}, 97:4--14,
  2014.

\bibitem{ChLiMNEE}
Wenxiong Chen and Congming Li.
\newblock {\em Methods on nonlinear elliptic equations}, volume~4 of {\em AIMS
  Series on Differential Equations \& Dynamical Systems}.
\newblock American Institute of Mathematical Sciences (AIMS), Springfield, MO,
  2010.

\bibitem{ChenLiCPAM2006}
Wenxiong Chen, Congming Li, and Biao Ou.
\newblock Classification of solutions for an integral equation.
\newblock {\em Comm. Pure Appl. Math.}, 59(3):330--343, 2006.

\bibitem{DancerYan1999}
E.~N. Dancer and Shusen Yan.
\newblock Multipeak solutions for a singularly perturbed {N}eumann problem.
\newblock {\em Pacific J. Math.}, 189(2):241--262, 1999.

\bibitem{DancerYan2007}
E.~N. Dancer and Shusen Yan.
\newblock A new type of concentration solutions for a singularly perturbed
  elliptic problem.
\newblock {\em Trans. Amer. Math. Soc.}, 359(4):1765--1790, 2007.

\bibitem{DavilaDelpinoSire2013}
Juan D{\'a}vila, Manuel del Pino, and Yannick Sire.
\newblock Nondegeneracy of the bubble in the critical case for nonlocal
  equations.
\newblock {\em Proc. Amer. Math. Soc.}, 141(11):3865--3870, 2013.

\bibitem{DiNezzaPalatucciValdinoci2012}
Eleonora Di~Nezza, Giampiero Palatucci, and Enrico Valdinoci.
\newblock Hitchhiker's guide to the fractional {S}obolev spaces.
\newblock {\em Bull. Sci. Math.}, 136(5):521--573, 2012.

\bibitem{DjadliHebeyLedoux2000}
Zindine Djadli, Emmanuel Hebey, and Michel Ledoux.
\newblock Paneitz-type operators and applications.
\newblock {\em Duke Math. J.}, 104(1):129--169, 2000.

\bibitem{ElgartSchleinCPAM2007}
Alexander Elgart and Benjamin Schlein.
\newblock Mean field dynamics of boson stars.
\newblock {\em Comm. Pure Appl. Math.}, 60(4):500--545, 2007.

\bibitem{FangGonzalezpreprint}
Yi~Fang and Maria del~Mar Gonz{\'a}lez.
\newblock {A}symptotic behavior of {P}alais-{S}male sequences associated with
  fractional {Y}amabe type equations.
\newblock Preprint (2014).

\bibitem{FrohlichLenzmanCPAM2007}
J{\"u}rg Fr{\"o}hlich and Enno Lenzmann.
\newblock Blowup for nonlinear wave equations describing boson stars.
\newblock {\em Comm. Pure Appl. Math.}, 60(11):1691--1705, 2007.

\bibitem{GonzalezJGA2012}
Maria del~Mar Gonz{\'a}lez, Rafe Mazzeo, and Yannick Sire.
\newblock Singular solutions of fractional order conformal {L}aplacians.
\newblock {\em J. Geom. Anal.}, 22(3):845--863, 2012.

\bibitem{GonzalezAPDE2013}
Mar{\'{\i}}a del~Mar Gonz{\'a}lez and Jie Qing.
\newblock Fractional conformal {L}aplacians and fractional {Y}amabe problems.
\newblock {\em Anal. PDE}, 6(7):1535--1576, 2013.

\bibitem{GrahamJenneSparling1992}
C.~Robin Graham, Ralph Jenne, Lionel~J. Mason, and George A.~J. Sparling.
\newblock Conformally invariant powers of the {L}aplacian. {I}. {E}xistence.
\newblock {\em J. London Math. Soc. (2)}, 46(3):557--565, 1992.

\bibitem{GramZworski2003}
C.~Robin Graham and Maciej Zworski.
\newblock Scattering matrix in conformal geometry.
\newblock {\em Invent. Math.}, 152(1):89--118, 2003.

\bibitem{JinLiyanyanXiongNirenbergproblemUnifiedapproach}
Tianling Jin, YanYan Li, and Jingang Xiong.
\newblock The {N}irenberg problem and its generalizations: {A} unified
  approach.
\newblock arXiv:1411.5743.

\bibitem{JinLiyanyanXiongNirenbergproblemBlowupanalysisII}
Tianling Jin, YanYan Li, and Jingang Xiong.
\newblock On a fractional nirenberg problem, part {II}: Existence of solutions.
\newblock {\em International Mathematics Research Notices}, 2013.

\bibitem{JinLiyanyanXiongNirenbergproblemBlowupanalysis}
Tianling Jin, YanYan Li, and Jingang Xiong.
\newblock On a fractional {N}irenberg problem, part {I}: blow up analysis and
  compactness of solutions.
\newblock {\em J. Eur. Math. Soc. (JEMS)}, 16(6):1111--1171, 2014.

\bibitem{KenigMartelRobiano2011}
C.~E. Kenig, Y.~Martel, and L.~Robbiano.
\newblock Local well-posedness and blow-up in the energy space for a class of
  {$L^2$} critical dispersion generalized {B}enjamin-{O}no equations.
\newblock {\em Ann. Inst. H. Poincar\'e Anal. Non Lin\'eaire}, 28(6):853--887,
  2011.

\bibitem{LiYanyanJEMS2005}
Yan~Yan Li.
\newblock Remark on some conformally invariant integral equations: the method
  of moving spheres.
\newblock {\em J. Eur. Math. Soc. (JEMS)}, 6(2):153--180, 2004.

\bibitem{LiYanyanNonlinearAnalysis2012}
YanYan Li, Paolo Mastrolia, and Dario~D. Monticelli.
\newblock On conformally invariant equations on {${\bf R}^n$}-{II}.
  {E}xponential invariance.
\newblock {\em Nonlinear Anal.}, 75(13):5194--5211, 2012.

\bibitem{LiYanyanNonlinearAnalysis2014}
YanYan Li, Paolo Mastrolia, and Dario~D. Monticelli.
\newblock On conformally invariant equations on {$\bold{R}^n$}.
\newblock {\em Nonlinear Anal.}, 95:339--361, 2014.

\bibitem{LiebAnnals1983}
Elliott~H. Lieb.
\newblock Sharp constants in the {H}ardy-{L}ittlewood-{S}obolev and related
  inequalities.
\newblock {\em Ann. of Math. (2)}, 118(2):349--374, 1983.

\bibitem{LiebYau1987}
Elliott~H. Lieb and Horng-Tzer Yau.
\newblock The {C}handrasekhar theory of stellar collapse as the limit of
  quantum mechanics.
\newblock {\em Comm. Math. Phys.}, 112(1):147--174, 1987.

\bibitem{MajdaMclaughlinTabak1997}
A.~J. Majda, D.~W. McLaughlin, and E.~G. Tabak.
\newblock A one-dimensional model for dispersive wave turbulence.
\newblock {\em J. Nonlinear Sci.}, 7(1):9--44, 1997.

\bibitem{PalatucciNA2015}
Giampiero Palatucci and Adriano Pisante.
\newblock A {G}lobal {C}ompactness type result for {P}alais-{S}male sequences
  in fractional {S}obolev spaces.
\newblock arXiv:1412.8392.

\bibitem{PalatucciCVPDE2014}
Giampiero Palatucci and Adriano Pisante.
\newblock Improved {S}obolev embeddings, profile decomposition, and
  concentration-compactness for fractional {S}obolev spaces.
\newblock {\em Calc. Var. Partial Differential Equations}, 50(3-4):799--829,
  2014.

\bibitem{Paneitz2008}
Stephen~M. Paneitz.
\newblock A quartic conformally covariant differential operator for arbitrary
  pseudo-{R}iemannian manifolds (summary).
\newblock {\em SIGMA Symmetry Integrability Geom. Methods Appl.}, 4:Paper 036,
  3, 2008.

\bibitem{PengshuangjieZhouJing2010}
Shuangjie Peng and Jing Zhou.
\newblock Concentration of solutions for a {P}aneitz type problem.
\newblock {\em Discrete Contin. Dyn. Syst.}, 26(3):1055--1072, 2010.

\bibitem{Peterson2000}
Lawrence~J. Peterson.
\newblock Conformally covariant pseudo-differential operators.
\newblock {\em Differential Geom. Appl.}, 13(2):197--211, 2000.

\bibitem{QingRaske2006}
Jie Qing and David Raske.
\newblock On positive solutions to semilinear conformally invariant equations
  on locally conformally flat manifolds.
\newblock {\em Int. Math. Res. Not.}, pages Art. ID 94172, 20, 2006.

\bibitem{ReyOlivier1990}
Olivier Rey.
\newblock The role of the {G}reen's function in a nonlinear elliptic equation
  involving the critical {S}obolev exponent.
\newblock {\em J. Funct. Anal.}, 89(1):1--52, 1990.

\bibitem{SavinValdini2012}
Ovidiu Savin and Enrico Valdinoci.
\newblock {$\Gamma$}-convergence for nonlocal phase transitions.
\newblock {\em Ann. Inst. H. Poincar\'e Anal. Non Lin\'eaire}, 29(4):479--500,
  2012.

\bibitem{Weistein1987}
Michael~I. Weinstein.
\newblock Existence and dynamic stability of solitary wave solutions of
  equations arising in long wave propagation.
\newblock {\em Comm. Partial Differential Equations}, 12(10):1133--1173, 1987.

\bibitem{Yanshusen2000}
Shusen Yan.
\newblock Concentration of solutions for the scalar curvature equation on
  {$\bold R^N$}.
\newblock {\em J. Differential Equations}, 163(2):239--264, 2000.

\end{thebibliography}
